\documentclass[11pt]{article}
\usepackage[top=2cm, bottom=2.5cm, left=2.5cm, right=2.5cm] {geometry}

\usepackage{comment}
\usepackage{amsmath,amssymb,theorem,color}
\numberwithin{equation}{section} 
\usepackage{amssymb}
\usepackage{color}
\usepackage{amsmath}
\usepackage{graphicx}
\usepackage{amsfonts}%
\newcommand{\R}{\ensuremath{\mathbb{R}}}

\newcommand{\N}{\ensuremath{\mathbb{N}}}
\newcommand{\eps}{\varepsilon}

\newcommand{\cC}{\mathcal{C}}
\newcommand{\cF}{\mathcal{F}}

\newcommand{\mP}{\mathbb{P}}
\newcommand{\Z}{\mathbb{Z}}

\newcommand{\cD}{\mathcal{D}}

\newcommand{\ltn}{\ensuremath{\left| \! \left| \! \left|}}
\newcommand{\rtn}{\ensuremath{\right| \! \right| \! \right|}}

\newtheorem{theorem}{Theorem}[section]
{ \theorembodyfont{\normalfont} 
	\newtheorem{example}[theorem]{Example}
	\newtheorem{remark}[theorem]{Remark}
}

\newtheorem{definition}[theorem]{Definition}
\newtheorem{lemma}[theorem]{Lemma}
\newtheorem{corollary}[theorem]{Corollary}

\newcounter{enumctr}
\newenvironment{enum}{\begin{list}{(\roman{enumctr})}{\usecounter{enumctr}}}{\end{list}}

\usepackage[displaymath, mathlines]{lineno}

\begin{document}

\title{Asymptotic stability for stochastic dissipative systems with a H\"older noise}

\author{Luu Hoang Duc\thanks{Max Planck Institute for Mathematics in the Sciences,
Inselstr. 22, 04103 Leipzig, Germany,
\& Institute of Mathematics, Viet Nam Academy of Science and Technology,
Hoang Quoc Viet str. 18, 10307 Ha Noi, Viet Nam
{\tt\small duc.luu@mis.mpg.de, lhduc@math.ac.vn}
}, $\;$ Phan Thanh Hong
\thanks{Thang Long University, Hanoi, Vietnam
{\tt\small hongpt@thanglong.edu.vn}
}, $\;$ 
Nguyen Dinh Cong
\thanks{Institute of Mathematics, Viet Nam Academy of Science and Technology,
Hoang Quoc Viet str. 18, 10307 Ha Noi, Viet Nam
{\tt\small ndcong@math.ac.vn}
}
}

\date{}
\maketitle

\begin{abstract}
    We prove the exponential stability of the zero solution of a stochastic differential equation with a H\"older noise, under the strong dissipativity assumption. As a result, we also prove that there exists a random pullback attractor for a stochastic system under a multiplicative fractional Brownian noise.   
\end{abstract}

\textbf{Keywords}: fractional Brownian motion, stochastic differential equations (SDE), Young integral, exponential stability, random attractor.

\section{Introduction}

In this paper we study the long term asymptotic behavior of the following nonautonomous stochastic differential equation 
\begin{equation}\label{linfBm}
dx(t) = [A(t) x(t) + F(t,x(t))]dt + C(t) x(t) dZ(t),\ x(0) = x_0 \in \R^d,
\end{equation}
where $Z(t)$ is a stationary stochastic process with almost sure all trajectories $\omega(t) = Z(t,\omega)$ to be H\"older continuous of index $\nu > \frac{1}{2}$. System \eqref{linfBm} can be solved by the pathwise approach with the help of Young integral \cite{young}. We will derive sufficient conditions on coefficient functions $A, F, C$, for which the zero solution is asymptotically or exponentially stable. 

Stochastic stability is systematically treated in \cite{khasminskii} and \cite{mao}. For example, the stability problem for system under a standard Brownian noise, i.~e.\ the case of which $Z(t)$ is replaced by the stochastic Brownian motion $B(t)$, can be studied using the Ito's formula
\[
d \|x(t)\|^2 
= \Big(2 \langle x(t), A(t) x(t) \rangle +  2 \langle x(t), F(t,x(t)) \rangle + \|C(t) x(t)\|^2 \Big) dt + 2 \langle x(t), C(t) x( t)\rangle dB(t),
\]
which follows that 
\begin{eqnarray}\label{Itoest}
d E \|x(t)\|^2 = E \Big(2 \langle x(t), A(t) x(t) \rangle +  2 \langle x(t), F(t,x(t)) \rangle + \|C(t) x(t)\|^2 \Big) dt, 
\end{eqnarray}
where $E$ denotes the expectation function.
Therefore under conditions on negative definiteness of $A(t)$ and global Lipschitz continuity of $F$ w.r.t. $x$ with a small Lipschitz constant, given $\|C(t)\|$ small enough, the quantity $ E \|x(t)\|^2 $ is exponentially decaying to zero, which implies that $\|x(t)\|$ converges exponentially and almost surely to zero due to Borel-Catelli lemma (see \cite[p\ 255]{Shiryaev}). 

The situation is however different here with equation \eqref{linfBm}, since in general $Z$ is neither a Markov process nor a semimartingale (e.g. fractional Brownian motion $B^H$ \cite{nourdin}), hence the expectation $E \langle x(t), C(t) x(t) \rangle dZ(t)$ does not vanish. Therefore a new approach to study stochastic stability is necessary.  Recently, the global dynamics is studied in \cite{ducGANSch18} for which the noise is assumed to be fractional Brownian motion with small noise in the sense that the H\"older seminorm of its realization is integrable and can be controlled to be small. On the other hand, the local stability is studied in \cite{garrido-atienzaetal} and in \cite{GABSch18} for which the diffusion coefficient $C(t) x(t)$ is replaced by $G(x(t))$ which is flat, i.e. $G(0) = DG(0) = 0$. It is also important to note that all above mentioned references apply fractional calculus (see also \cite{mishura}, \cite{nualart}, \cite{zahle}, \cite{zahle2}) and the semigroup approach to deal with the stability problem. 

Looking back at the classical theory of ordinary differential equations we know that there are two fundamental methods to deal with stability problem of solution of an ODE --- the methods of Lyapunov, which proved to be powerful tools of qualitative theory of ODE and the stability theory in particular.  In case of the first method one linearizes the system near an equilibrium and studies the growth rate (Lyapunov exponents) of the solutions and the spectrum of derived linear system and then deduces the asymptotic properties of the original nonlinear systems near the fixed point.  In case of the second Lyapunov method one studies the action of the ODE on a specific function (called Lyapunov function) and then deduces asymptotic properties of the system without the need of solving the ODE explicitly (hence this method is called the method of Lyapunov functions).

In this paper we reinvestigate the stability problem using a different method compared to the references mentioned above, namely we use the approach of the second Lyapunov method: we construct a Lyapunov-type function, which is the norm function, and combine the discretization scheme developed in \cite{congduchong17}, \cite{congduchong18} and \cite{ducGANSch18} but for polar coordinates, using $p {\rm -var}$ norm estimates. The main difficulty lies in how to use path-wise estimates to deal with the driving noise, which is expected to be technical. We prove in Theorem \ref{stabfSDElin} that for $A$ negative definite and $F$ with small Lipschitz coefficient, one can choose $C$ small enough in terms of average $q-$var norm such that the system is pathwise exponentially stable. As such, the result gives a significantly better stability criterion than those in \cite{ducGANSch18} and \cite{GAKLBSch2010}, and moreover matches the stability criteria for ordinary differential equations when the noise is diminished (see details in Remark \ref{comparison}). To our knowledge, our method is also the first attempt to study the stability for Young differential equations using Lyapunov type functions. 

The result is then applied to study the asymptotic behavior of the stochastic system
\begin{equation}\label{fSDE0}
dx(t) = [Ax(t) + f(x(t))]dt + Cx(t)dB^H(t),t\in \R,\ x(0)=x_0 \in \R^d, 
\end{equation}
where we assume for simplicity that $A,C \in \R^{d\times d}$, $f: \R^d \to \R^d$ such that $f(0) \ne 0$, and $B^H$ is an one-dimensional fractional Brownian motion with Hurst exponent $H \in (1/2,1)$ \cite{mandelbrot}, i.e. it is a family of centered Gaussian processes $B^H = \{B^H(t)\}$, $t\in \R$ with continuous sample paths and the covariance function
\[
R_H(s,t) = \tfrac{1}{2}(t^{2H} + s^{2H} - |t-s|^{2H}), \forall t,s \in \R.
\]
Since no deterministic equilibrium such as the zero solution is found, system \eqref{fSDE0} is expected to possess a random attractor, which is a generalization of the classical attractor concept (see e.g. \cite{crauel-flandoli} or \cite{crauelkloeden15} for a survey on random attractor theory). In the stochastic setting with fractional Brownian motions, in \cite{GAKLBSch2010}  the existence of the random attractor is investigated assuming that the diffusion coefficient is bounded. Here in this paper, we will prove in Theorem \ref{attractor} that there exists a global random attractor for system \eqref{fSDE0}, and moreover the random attractor consists of only one random point.

\section{Preliminaries}
\subsection{Young integral}
Let $C([a,b],\R^d)$ denote the space of all continuous paths $x:\;[a,b] \to \R^d$ equipped with sup norm $\|\cdot\|_{\infty,[a,b]}$ given by $\|x\|_{\infty,[a,b]}=\sup_{t\in [a,b]} \|x(t)\|$, where $\|\cdot\|$ is the Euclidean norm in $\R^d$. For $p\geq 1$ and $[a,b] \subset \R$, $\cC^{p{\rm-var}}([a,b],\R^d)\subset C([a,b],\R^d)$ denotes the space of all continuous paths $x:[a,b] \to \R^d$ which are of finite $p-$variation 
\begin{eqnarray}
\ltn x\rtn_{p\rm{-var},[a,b]} :=\left(\sup_{\Pi(a,b)}\sum_{i=1}^n \|x(t_{i+1})-x(t_i)\|^p\right)^{1/p} < \infty,
\end{eqnarray}
where the supremum is taken over the whole class of finite partitions of $[a,b]$. $\cC^{p{\rm-var}}([a,b],\R^d)$ equipped with the $p-$var norm
\begin{eqnarray*}
	\|x\|_{p\text{-var},[a,b]}&:=& \|x(a)\|+\ltn x\rtn_{p\rm{-var},[a,b]},
\end{eqnarray*}
is a nonseparable Banach space \cite[Theorem 5.25, p.\ 92]{friz}. Also for each $0<\alpha<1$, we denote by $C^{\alpha\rm{-Hol}}([a,b],\R^d)$ the space of H\"older continuous functions with exponent $\alpha$ on $[a,b]$ equipped with the norm
\[
\|x\|_{\alpha\rm{-Hol},[a,b]}: = \|x(a)\| + \sup_{a\leq s<t\leq b}\frac{\|x(t)-x(s)\|}{(t-s)^\alpha}.
\]

Given a simplex $\Delta[a,b]:=\{(s,t)|\  a\leq s\leq t\leq b\}$, a continuous map $\overline{\omega}: \Delta[a,b]\longrightarrow \R^+$ is called a {\it control} (see e.g. \cite{friz}) if it is zero on the diagonal and superadditive, i.e\\
(i), For all $t\in [a,b]$, $\overline{\omega}_{t,t}=0$,\\
(ii), For all $s\leq t\leq u$ in $[a,b]$, $\overline{\omega}_{s,t}+\overline{\omega}_{t,u}\leq \overline{\omega}_{s,u}$.

Now, consider $x\in \cC^{q{\rm-var}}([a,b],\R^{d\times m})$ and $\omega\in \cC^{p{\rm -var}}([a,b],\R^m)$ with  $\frac{1}{p}+\frac{1}{q}  > 1$, the Young integral $\int_a^bx(t)d\omega(t)$ can be defined as 
\[
\int_a^bx(s)d\omega(s):= \lim \limits_{|\Pi| \to 0} \sum_{[u,v] \in \Pi} x(u)(\omega(v)-\omega(u)) ,
\]
where the limit is taken on all the finite partitions $\Pi=\{ a=t_0<t_1<\cdots < t_n=b \}$ of $[a,b]$ with $|\Pi| := \displaystyle\max_{[u,v]\in \Pi} |v-u|$ (see \cite[p.\ 264--265]{young}). This integral satisfies additive property by the construction, and the so-called Young-Loeve estimate \cite[Theorem 6.8, p.\ 116]{friz}
\begin{eqnarray}\label{YL0}
\Big\|\int_s^t x(u)d\omega(u)-x(s)[\omega(t)-\omega(s)]\Big\| \leq K \ltn x\rtn_{q\rm{-var},[s,t]} \ltn\omega\rtn_{p\rm{-var},[s,t]}, \;\forall \; [s,t]\subset [a,b]
\end{eqnarray}
where 
\begin{equation}\label{constK}
K:=(1-2^{1-\theta})^{-1},\qquad \theta := \frac{1}{p} + \frac{1}{q} >1.
\end{equation}
Throughout this paper, we would assume for simplicity that $m=1$. Notice that all the results are still correct for any $m \in \N$, with a small modification. 
\subsection{Nonlinear Young differential equations}

For any fixed $1<p <2$, $T>0$ and a continuous path $\omega$ that belongs to $ \cC^{p{\rm-var}}([0,T],\R)$, consider the deterministic differential equation in the Young sense
\begin{equation}\label{lin1}
dx(t) = [A(t) x(t) + F(t,x(t))]dt  + C(t)x(t) d\omega(t),\;\; x(0)=x_0,
\end{equation}
where $0\leq t\leq T$, $x_0\in\R^d$, $A\in C([0,T],\R^{d\times d})$ and $C \in \cC^{q{\rm-var}}([0,T],\R^{d\times d})$ with $q$ satisfying $q\geq p$ and $\frac{1}{p} + \frac{1}{q} >1$. Additionally, $F$ is globally Lipschitz continuous w.r.t. $x$, i.e there exists $L>0$ such that for all $t\in[0,T]$, for all  $x,y\in \R^d$: $\|F(t,x)-F(t,y)\|\leq L\|x-y\|$. Then the system \eqref{lin1} possesses a unique solution  in both the forward and backward sense, as studied in \cite{congduchong17,congduchong18}. In fact under these conditions the system  can be transformed to a classical ordinary differential equation which satisfies the existence and uniqueness theorem. 
\begin{theorem}\label{existence}
	There exists a unique solution to the system \eqref{lin1} in the space $\cC^{q{\rm-var}}([0,T],\R^{d})$.
\end{theorem}
\begin{proof}
	Indeed, due to  \cite{congduchong18}, there exists a unique solution to the equation 
	\begin{equation}\label{lin1a}
	dz(t) = A(t) z(t)dt  + C(t)z(t) d\omega(t)
	\end{equation}
	in the space $ \cC^{q{\rm-var}}([0,T],\R^{d})$. Denote by $\Phi(t,\omega)$ the fundamental matrix of solution of \eqref{lin1a} with $\Phi(0,\omega) = Id$ - the identity matrix. Put $u(t) = \Phi^{-1}(t,\omega)x(t)$, then by the integration by part formula, $u$ satisfies the equation
	\begin{eqnarray}\label{u}
	du(t)&=& \Phi^{-1}(t,\omega) dx(t) + d\Phi^{-1}(t,\omega)x(t)\notag\\
	&=&\Phi^{-1}(t,\omega) \Big[\Big(A(t)x(t) +F(t,x(t))\Big)dt  + C(t)x(t)d\omega(t)\Big] \notag\\
	&&-\Phi^{-1}(t,\omega)\Big(A(t)\Phi(t) dt + C(t) \Phi(t,\omega)d\omega(t)\Big)\Phi^{-1}(t,\omega)x(t)\notag\\
	&=& \Phi^{-1}(t,\omega)F(t,\Phi(t,\omega)u(t))dt=:G(t,u(t))dt.
	\end{eqnarray}
	Since, $\Phi(\cdot,\omega)$ and $\Phi^{-1}(\cdot,\omega)$ are continuous on $[0,T]$, it is easy to check that $G(t,u)$ satisfy the global Lipschitz condition which assures the existence and uniqueness of a global solution to \eqref{u} on $[0,T]$, and moreover $u\in C^{1}([0,T],\R^d)$. The one-one correspondence between solutions of \eqref{lin1} and solutions of \eqref{u} then prove the existence and uniqueness of solution of \eqref{lin1}. The same conclusion holds for the backward equation of \eqref{lin1}.
\end{proof}
\section{Exponential stability of nonlinear Young differential equations}

In this section we are going to study the exponential stability of \eqref{lin1} where $\omega\in \cC^{p{\rm-var}}([0,T],\R)$, $A\in C([0,T],\R^{d\times d})$ and $C \in \cC^{q{\rm-var}}([0,T],\R^{d\times d})$ for any $T>0$. First, we formulate the definition of stability for deterministic Young differential equations (for the classical stability notion see e.g. \cite[p.~17]{hale}, \cite[p.~152]{nemytskii}, or \cite{ducetal06}).
\begin{definition}\label{Defstability}
	(A) Stability: A solution $\mu(\cdot)$ of the deterministic Young differential equation \eqref{lin1} is called stable, if for any $\varepsilon >0$ there exists an $r =r(\varepsilon)>0$ such that for any solution $x(\cdot)$ of  \eqref{lin1} satisfying  $\|x(0)-\mu(0)\| < r$ the following inequality holds 
	\[
	\sup_{t\geq 0}\|x(t)-\mu(t)\| < \varepsilon.
	\]
	(B) Attractivity: $\mu(\cdot)$ is called attractive, if  there exists $r >0$ such that for any solution $x(\cdot)$ of  \eqref{lin1} satisfying  $\|x(0)-\mu(0)\| < r$ we have
	\[
	\lim \limits_{t \to \infty} \|x(t)-\mu(t)\| = 0.   
	\]	
	(C) Asymptotic stability: $\mu(\cdot)$ is called
	\begin{enum}
		\item asymptotically stable, if it is stable and attractive.
		\item exponentially  stable, if it is stable and there exists $r>0$ such that for any solution $x(\cdot)$ of  \eqref{lin1} satisfying  $\|x(0)-\mu(0)\| < r$ we have
		\[
		\varlimsup \limits_{t \to \infty} \frac{1}{t} \log\|x(t)-\mu(t)\|
		< 0. 
		\]
	\end{enum}
\end{definition}
Below we need several assumptions for $A, F,C$.\\

(${\textbf H}_1$) $A$ is negative definite in the sense  that there exists a  function $h: \R^+ \to \R^+$ such that
\begin{equation}\label{negdefh}
\langle x,A(t)x \rangle \leq -h(t) \|x\|^2,\quad\hbox{for all}\quad  x\in \R^d.
\end{equation} 

(${\textbf H}_2$) $F(t,0) \equiv 0$ for all $t\in \R^+$ and $F(t,x)$ is of globally Lipschitz continuous w.r.t. $x$, i.e. there exists a positive continuous function $f: \R^+ \to \R^+$ such that
\begin{equation}\label{lipF}
\|F(t,x)-F(t,y)\| \leq f(t) \|x-y\|, \quad \forall x, y \in \R^d.
\end{equation}

(${\textbf H}_3$) There exist constants 
\begin{eqnarray}
&&\hat{A}:=\varlimsup_{m\to\infty}\left(\frac{1}{m+1}\sum_{k=0}^m \Big(\|A\|_{\infty,\Delta_k} + \|f\|_{\infty,\Delta_k}\Big)^{4p} \right)^{\frac{1}{4p}}< \infty; \label{A-hat}\\	
&&\hat{C}:= \varlimsup \limits_{m \to \infty}\left(\frac{1}{m+1} \sum_{k=0}^{m}\|C\|_{q{\rm-var},  \Delta_k}^{2p+2}\right)^{\frac{1}{2p+2}}< \infty; \label{C-hat}\\
&&\Gamma(\omega,2p+2):=\varlimsup_{m\to\infty}\left(\frac{1}{m+1}\sum_{k=0}^m \ltn \omega\rtn^{2p+2}_{p{\rm- var},\Delta_k}\right)^{\frac{1}{2p+2}} < \infty, \label{omegagrowth}
\end{eqnarray}
where $\Delta_k:=[k,k+1]$.
\begin{remark}\label{h3test}
	(i), Since $\langle x,A(t)x\rangle  = \frac{1}{2} \langle x,A(t)x\rangle  + \frac{1}{2} \langle  x, A^T(t)x \rangle  =\langle x,B(t)x\rangle$, where $B(t) = \frac{1}{2}[A(t)+A^T(t)]$
	and since the smallest eigenvalue $h^*(t)$ of the symmetric matrix $-B(t)$ satisfies
	$$h^*(t) = \min\{ \langle x,-B(t)x\rangle \mid \ \|x\|=1\},$$
	it follows from (${\textbf H}_1$) that $h^*(t)\geq h(t)$ for all $t\in\R^+$, $h$ can also be replaced by $h^*$ in asssumption (${\textbf H}_1$). The reader is referred to \cite{demidovich}, \cite{wazewski} for stability theory of ordinary differential equations.    
	
	(ii) While assumptions (${\textbf H}_1$) and (${\textbf H}_2$) are usual, it is important to note that (${\textbf H}_3$) is satisfied in the simplest case of autonomous systems, i.e. $A(t) \equiv A, C(t) \equiv C$ and $f$ is bounded on $\R^+$. Then $\hat{A} \leq \|A\| + \|f\|_{\infty,\R^+}, \hat{C}= \|C\|$. For a nontrivial example, consider $A(t) = A(\Theta_t \eta), f(t) = f(\Theta_t \eta), C(t) = C(\Theta_t \eta)$ which depends on a dynamical system $\Theta_t$ on a space of elements $\eta \in C^{q{\rm-var}}$ such that $\Theta$ is invariant under some probability measure. Then $A(\cdot), C(\cdot)$ are functions of a stationary process. Conditions \eqref{A-hat} and \eqref{C-hat} are equivalent to 
	\begin{eqnarray}
	\hat{A}= \Big[ E (\|A(\eta)\|_{\infty,[0,1]}+ \|f(\eta)\|_{\infty,[0,1]})^{4p} \Big]^{\frac{1}{4p}}< \infty,\\ 
	\hat{C}= \left(E \|C(\eta)\|^{2p+2}_{q{\rm-var},[0,1]} \right)^{\frac{1}{2p+2}}< \infty.
	\end{eqnarray}
	Meanwhile, assumption \eqref{omegagrowth} is satisfied for almost sure all trajectories $\omega$ of the stationary process $Z(t)$ if
	\begin{equation}\label{gammaome}
	\Gamma(\omega,2p+2) = \Big(E (\ltn Z(\cdot) \rtn_{p {\rm-var},[0,1]}^{2p+2})\Big)^{\frac{1}{2p+2}} < \infty.
	\end{equation}
	
	(iii) It is easy to check (see \cite{congduchong17} and \cite{congduchong18}) that conditions (${\textbf H}_2$) and (${\textbf H}_3$) assure the existence and uniqueness of a global solution to \eqref{lin1} on $\R^+$.
\end{remark}
\begin{lemma}\label{lem70}
	Let $1\leq p\leq q$ be arbitrary and satisfy $\frac{1}{p}+\frac{1}{q}>1$.  Assume that $\omega \in \cC^{p{\rm-var}}([0,T],\R)$ and $y\in \cC^{q{\rm-var}}([0,T],\R^d)$ satisfy
	\begin{equation}\label{condition2}
	\ltn y\rtn_{q{\rm-var},[s,t]}\leq b(1 +\ltn y\rtn_{q{\rm-var},[s,t]}) (t-s+\ltn \omega\rtn_{p{\rm-var},[s,t]}), 
	\end{equation}
	for all $[s,t] \subset [0,T]$, where $b\geq 0$ is a constant. Then there exists a constant $C(b)$ independent of $T$ such that the following inequality holds for every $s<t$ in $[0,T]$ 
	\begin{eqnarray}\label{yest}
	\ltn y\rtn_{q{\rm-var},[s,t]} 
	&\leq& C(b) \max \Big\{(t-s)^p+\ltn\omega\rtn^p_{p{\rm-var},[s,t]}, (t-s)+ \ltn\omega\rtn_{p{\rm-var},[s,t]} \Big\}.
	\end{eqnarray}
\end{lemma}
\begin{proof}
	Set $\overline{\omega}(s,t) = 2^{2p-1}b^p[(t-s)^p+\ltn \omega\rtn^p_{p\rm{-var},[s,t]}]$, then $\overline{\omega}(s,t)$ is a control on $\Delta[0,T]$ (see \cite{friz}) and due to the inequality $(a+b)^r\leq (a^r+b^r)\max\{1,2^{r-1}\},\;\forall a>0,b>0,r>0$ we have
	$$\ltn y\rtn_{q{\rm-var},[s,t]}\leq \frac{1}{2}(1+\ltn y\rtn_{q{\rm-var},[s,t]})\overline{\omega}(s,t)^{1/p}.$$
	This implies that 
	$$|y(t)-y(s)|\leq \ltn y\rtn_{q{\rm-var},[s,t]}\leq \overline{\omega}(s,t)^{1/p}$$
	for all $s,t\in [0,T]$ such that $\overline{\omega}(s,t)\leq 1$.
	Due to Proposition 5.10 of \cite{friz}, we have
	\begin{eqnarray*}
		\ltn y\rtn_{q{{\rm -var}},[s,t]}&\leq& 2 \max\{\overline{\omega}(s,t)^{1/p}, \overline{\omega}(s,t)\}\\
		&\leq& C(b) \max \Big\{(t-s)^p+\ltn\omega\rtn^p_{p{\rm-var},[s,t]}, (t-s)+ \ltn\omega\rtn_{p{\rm-var},[s,t]} \Big\},
	\end{eqnarray*}
	in which $C(b)=2\max\{(4b)^p,4b\}$.\\
\end{proof}
Our first main result on stability of system \eqref{lin1} can be formulated as follows.
\begin{theorem}\label{stabfSDElin}
	Suppose that the conditions (${\textbf H}_1$) -- (${\textbf H}_3$) are satisfied,  and further that 
	\begin{equation}\label{negdef}
	\liminf \limits_{t \to \infty} \frac{1}{t} \int_0^t [h(s)-f(s)]ds \geq h_0 >0.
	\end{equation} 
	Then under the condition 
	\begin{eqnarray}\label{criterion}
	h_0&>&K\Big(1 +4\hat{G}\Big)\hat{C}\Big[\Gamma(\omega,2) +  \Gamma(\omega,4)^{2}+\Gamma(\omega,2p+2)^{p+1}\Big] 
	\end{eqnarray}
	where $K$ is given by \eqref{constK} and 
	\[
	\hat{G}:=\max \Big \{8\hat{A}, 16K\hat{C}, 8^p\hat{A}^p, 16^pK^p \hat{C}^p \Big \},
	\]
	the zero solution of system \eqref{lin1} is exponentially stable.
\end{theorem}


\begin{proof}
	Our proof is divided into three steps. In {\bf Step 1}, we use polar coordinates to derive the growth rate of the solution in \eqref{trans3}. The estimate for $q-$var seminorm of the angular $y$ is then derived in \eqref{yqvarest} in {\bf Step 2}, applying Lemma \ref{lem70}. As such, the solution growth rate can finally be estimated in \eqref{rate}, in which each component is estimated in {\bf Step 3} using hypothesis (${\textbf H}_3$). The theorem is then proved by choosing $\epsilon$ such that \eqref{criterion} is satisfied.  \\
	
	{\bf Step 1:} Put $r(t) := \|x(t)\|$. Due to the fact that the system \eqref{lin1} possesses a unique solution  in both the forward and backward sense and that $x(t)\equiv 0$ is the unique solution through zero, the solution starting from the initial condition $x(0)\ne 0\in\R^d$ satisfies $x(t)\ne 0$ for all $t\in \R^+$. We then can define $y(t) := \frac{x(t)}{\|x(t)\|}$. Using integration by part technique (see, e.g., Z{\"a}hle~\cite{zahle,zahle2}), it is easy to prove that $r(t)$ satisfies the system 
	\allowdisplaybreaks
	\begin{eqnarray}\label{trans1}
	dr(t) &=& \Big[\langle y(t), A(t)y(t)\rangle + \langle y(t), \frac{F(t,x(t))}{\|x(t)\|}\rangle \Big] r(t)dt + \langle y(t), C(t)y(t)\rangle r(t) d\omega(t), 
	\end{eqnarray} 
	where 
	\allowdisplaybreaks
	\begin{eqnarray}\label{eqy}
	dy(t) &=& \frac{r(t) d x(t) - x(t) d r(t)}{r(t)^2} \nonumber \\
	&=& \Big[A(t)y(t) - y(t) \langle y(t), A(t)y(t)\rangle + \frac{F(t,x(t))}{\|x(t)\|} - y(t) \langle y(t), \frac{F(t,x(t))}{\|x(t)\|}\rangle \Big] dt \nonumber \\
	&&+ [C(t)y(t) - y(t) \langle y(t), C(t)y(t)\rangle] d\omega(t).
	\end{eqnarray}
	Again using the integration by parts, we can prove that
	\begin{eqnarray}\label{trans2}
	d\log r(t) &=& \Big[\langle y(t), A(t)y(t)\rangle + \langle y(t), \frac{F(t,x(t))}{\|x(t)\|}\rangle \Big] dt + \langle y(t), C(t)y(t)\rangle d\omega(t),
	\end{eqnarray}
	or in the integration form
	\allowdisplaybreaks
	\begin{eqnarray*}
	\log r(t) = \log r(0) + \int_0^t \Big[\langle y(s), A(s)y(s)\rangle + \langle y(s), \frac{F(s,x(s))}{\|x(s)\|}\rangle \Big] ds+ \int_0^t  \langle y(s), C(s)y(s)\rangle d\omega(s).
	\end{eqnarray*}
	Due to \eqref{lipF}, $\Big\|\frac{F(t,x)}{\|x\|}\Big\| \leq f(t)$ for any $x \ne 0$, hence
	\allowdisplaybreaks
	\begin{eqnarray}\label{trans3}
	&&\frac{1}{t}\log r(t) \notag\\
	&\leq& \frac{1}{t}\log r(0) + \frac{1}{t} \int_0^t \Big[\langle y(s), A(s)y(s)\rangle + \Big|\langle y(s), \frac{F(s,x(s))}{\|x(s)\|}\rangle \Big| \Big]ds + \frac{1}{t} \Big|\int_0^t  \langle y(s), C(s)y(s)\rangle d\omega(s)\Big| \nonumber\\
	&\leq& \frac{1}{t}\log r(0) - \frac{1}{t}\int_0^t [h(s)-f(s)]ds+ \frac{1}{t} \Big| \int_0^t \langle y(s), C(s)y(s)\rangle d\omega(s) \Big|.
	\end{eqnarray}
	
	{\bf Step 2:} To estimate the third term in the right hand side of \eqref{trans3}, we use the discretization scheme. Note that 
	\begin{eqnarray*}
		\Big| \int_{k}^{k+1} \langle y(s), C(s)y(s) \rangle d\omega(s)\Big| 
		&\leq&  \ltn \omega \rtn_{p{\rm-var},\Delta_k} \left( |\langle y(k),C(k) y(k)\rangle |+ K \ltn \langle y, Cy\rangle\rtn_{q{\rm-var},\Delta_k}\right)\\
		&\leq&  \ltn \omega \rtn_{p{\rm-var},\Delta_k} \left( \|C(k)\|+ K \ltn \langle y, Cy\rangle\rtn_{q{\rm-var},\Delta_k}\right), 
	\end{eqnarray*}
due to the fact that $\|y(t)\| = 1$ where 
	\begin{eqnarray*}
		\ltn\langle y, Cy \rangle\rtn_{q{\rm-var},\Delta_k} &\leq& \|y\|_{\infty,\Delta_k} \ltn Cy\rtn_{q{\rm-var},\Delta_k} + \ltn y\rtn_{q{\rm-var},\Delta_k} \|Cy\|_{\infty,\Delta_k}\\
		&\leq& \ltn C \rtn_{q{\rm-var},\Delta_k} \|y\|_{\infty,\Delta_k} + \|C\|_{\infty,\Delta_k} \ltn y\rtn_{q{\rm-var},\Delta_k} + \ltn y\rtn_{q{\rm-var},\Delta_k} \|C\|_{\infty,\Delta_k}\|y\|_{\infty,\Delta_k}\\	
		&\leq& 2 \| C \|_{\infty,\Delta_k} \ltn y \rtn_{q{\rm-var},\Delta_k} + \ltn C \rtn_{q{\rm-var},\Delta_k}.
	\end{eqnarray*}
	Hence,
	\begin{eqnarray}\label{qvarest}
	&&\Big| \int_{k}^{k+1} \langle y(s), C(s)y(s) \rangle d\omega(s) \Big| \notag\\
	&\leq& \ltn \omega\rtn_{p{\rm-var},\Delta_k} \Big(\|C(k)\|+ 2K \|C\|_{\infty,\Delta_k}\ltn y\rtn_{q{\rm-var},\Delta_k}+ K \ltn C\rtn_{q{\rm-var},\Delta_k}\Big) \nonumber \\
	&\leq&K \|C\|_{q{\rm-var},\Delta_k}\ltn\omega\rtn_{p{\rm-var},\Delta_k}\Big(1+  2\ltn y\rtn_{q{\rm-var},\Delta_k} \Big).
	\end{eqnarray}
	On the other hand, from \eqref{eqy} we derive that $y$ satisfies the equation:
	\begin{eqnarray*}
		y(t) -y(0) &=&\int_0^t \Big[A(s)y(s) - y(s) \langle y(s), A(s)y(s)\rangle + \frac{F(s,x(s))}{\|x(s)\|} - y(s) \langle y(s), \frac{F(s,x(s))}{\|x(s)\|}\rangle \Big] ds \\
		&&+ \int_0^t[C(s)y(s) - y(s) \langle y(s), C(s)y(s)\rangle] d\omega(s)\\
		&=: &I(y)(t)+J(y)(t), \quad \forall t\geq 0,
	\end{eqnarray*} 
	hence for all $0<a \leq b$
	\begin{eqnarray*}
		\ltn y\rtn_{q{\rm-var},[a,b]}\leq \ltn I(y) \rtn_{q{\rm-var},[a,b]}+ \ltn J(y)\rtn_{q{\rm-var},[a,b]}.
	\end{eqnarray*} 
	Since $\|y(t)\| = 1$, a direct computation shows that for $ 0\leq a<b$,
	\begin{equation*}\label{Iest}
	\ltn I(y)  \rtn_{q{\rm-var},[a,b]}\leq(b-a) \left(2\|A\|_{\infty,[a,b]} + 2 \|f\|_{\infty,[a,b]}\right), 
	\end{equation*}
	and 
	\begin{eqnarray*}\label{Jest}
	&&\ltn J(y) \rtn_{q{\rm-var},[a,b]}\\
	&\leq & K \ltn \omega\rtn_{p{\rm-var},[a,b]} \Big( \|Cy\|_{\infty,[a,b]}+ \| y \langle y, Cy\rangle\|_{\infty,[a,b]} + \ltn Cy\rtn_{q{\rm-var},[a,b]}  +\ltn y \langle y, Cy\rangle\rtn_{q{\rm-var},[a,b]} \Big) \nonumber\\
	&\leq & K \ltn \omega\rtn_{p{\rm-var},[a,b]}\Big(2 \|C\|_{\infty,[a,b]}+2\ltn C \rtn_{q{\rm-var},[a,b]}+ 4\|C\|_{\infty,[a,b]}\ltn y \rtn_{q{\rm-var},[a,b]}\Big) \nonumber\\
	&\leq & 4K \|C\|_{q{\rm-var},[a,b]}\ltn \omega\rtn_{p{\rm-var},[a,b]}(1+\ltn y \rtn_{q{\rm-var},[a,b]}) .
	\end{eqnarray*} 
	Put $\hat{A}_k:=\|A\|_{\infty,\Delta_k} + \|f\|_{\infty,\Delta_k}$ and $\hat{C}_k:=\|C\|_{q{\rm-var},  \Delta_k}, k\in \N$. Then for $[a,b]\subset \Delta_k$
	\begin{eqnarray*}
	\ltn y\rtn_{q{\rm-var},[a,b]} &\leq& \max\{2\hat{A}_k,4K\hat{C}_k\} \Big[(b-a)+\ltn \omega\rtn_{p{\rm-var},[a,b]} \Big]\Big(1+ \ltn y\rtn_{q{\rm-var},[a,b]}\Big).
	\end{eqnarray*} 
	By applying Lemma \ref{lem70} we obtain 
	\begin{eqnarray}\label{yqvarest}	
	\ltn y\rtn_{q{\rm-var},\Delta_k}&\leq& 2 G_k\max \Big\{ 1+\ltn \omega\rtn_{p{\rm-var},\Delta_k}, 1 +\ltn \omega\rtn^p_{p{\rm-var},\Delta_k}\Big\}\notag\\
	&\leq&  2 G_k\Big( 1+\ltn \omega\rtn_{p{\rm-var},\Delta_k}+\ltn \omega\rtn^p_{p{\rm-var},\Delta_k}\Big),
	\end{eqnarray}
	where
	\[
	G_k:= \max\Big\{8\hat{A}_k,16K\hat{C}_k, 8^p\hat{A}_k^p,16^pK^p\hat{C}_k^p\Big\}.
	\]
	For any $t \in [m,m+1]$,
	\begin{eqnarray}\label{trans4}
	&&\frac{1}{t}  \Big| \int_0^t \langle y(s), C(s)y(s)\rangle d\omega(s) \Big|\notag\\
	 &\leq& \frac{1}{m} \left(\sum_{k=0}^{m -1}\left | \int_{k}^{k+1} \langle y(s), C(s)y(s) \rangle d\omega(s)\right |+ \left | \int_{m}^{t} \langle y(s), C(s)y(s) \rangle d\omega(s)\right |\right) \notag\\
	&&\leq \frac{K}{m}  \sum_{k=0}^m \hat{C}_k\ltn \omega\rtn_{p{\rm-var},\Delta_k} \Big(1+2\ltn y\rtn_{q{\rm-var},\Delta_k} \Big).
	\end{eqnarray}
	Combining \eqref{trans4} with \eqref{trans3} and \eqref{yqvarest}, we get
	\begin{eqnarray}\label{rate}
	\varlimsup_{t\to \infty}\frac{1}{t}\log r(t) 
	&\leq& -h_0 +  \varlimsup_{m\to \infty} \frac{K}{(m+1)}  \sum_{k=0}^m \hat{C}_k\ltn \omega\rtn_{p{\rm-var},\Delta_k} \notag \\
	&&+ \varlimsup_{m\to \infty}\frac{4K}{(m+1)}  \sum_{k=0}^m  \hat{C}_k G_k \Big(\ltn \omega\rtn_{p{\rm-var},\Delta_k}+\ltn \omega\rtn^{2}_{p{\rm-var},\Delta_k}+ \ltn \omega\rtn_{p{\rm-var},\Delta_k}^{p+1}\Big).\notag\\
	\end{eqnarray}

	{\bf Step 3.} Using H\"older inequality, the second term in \eqref{rate} can be estimated as follows
	\begin{eqnarray*}
		\varlimsup_{m\to \infty} \frac{K}{(m+1)}  \sum_{k=0}^m \hat{C}_k\ltn \omega\rtn_{p{\rm-var},\Delta_k} &\leq &K\varlimsup_{m\to \infty} \Big(\frac{1}{m+1}  \sum_{k=0}^m \hat{C}^2_k\Big)^{\frac{1}{2}}  \varlimsup_{m\to \infty} \Big(\frac{1}{m+1}  \sum_{k=0}^m  \ltn\omega\rtn_{p{\rm-var},\Delta_k}^2 \Big)^{\frac{1}{2}} \\
		&\leq& K\hat{C} \Gamma(\omega,2).
	\end{eqnarray*}
	Similarly, we get the estimates for the other terms at the right hand side of \eqref{rate} so that
	\begin{eqnarray*}
		\varlimsup_{t\to \infty}\frac{1}{t}\log r(t) &\leq& - h_0 +K \hat{C}\Gamma(\omega,2) \\ 
		&&+4K\varlimsup_{m\to \infty} \Big(\frac{1}{m+1}  \sum_{k=0}^m  \hat{C}^2_k G_k^2\Big)^{\frac{1}{2}}\Big( \Gamma(\omega,2)+ \Gamma(\omega,4)^{2}+\Gamma(\omega,2p+2)^{p+1}\Big),
	\end{eqnarray*}
	where all the values of $\Gamma$ are finite due to assumption \eqref{omegagrowth}. To estimate the average of $\hat{C}^2_k G_k^2$, observe that
	\allowdisplaybreaks
	\begin{eqnarray*}
		&& \varlimsup_{m\to \infty} \left(\frac{1}{m+1}  \sum_{k=0}^m \hat{A}_k^2\hat{C}_k^2\right)^{\frac{1}{2}} \leq  \varlimsup_{m\to \infty} \left(\frac{1}{m+1}  \sum_{k=0}^m \hat{A}_k^4\right)^{\frac{1}{4}} \varlimsup_{m\to \infty} \left(\frac{1}{m+1}  \sum_{k=0}^m \hat{C}_k^4\right)^{\frac{1}{4}}\leq \hat{A}\hat{C};\\
		&&\varlimsup_{m\to \infty} \left(\frac{1}{m+1}  \sum_{k=0}^m \hat{C}_k^4\right)^{\frac{1}{2}}\leq  \hat{C}^2;\\
		&&\varlimsup_{m\to \infty} \left(\frac{1}{m+1}  \sum_{k=0}^m \hat{A}_k^{2p}\hat{C}_k^2\right)^{\frac{1}{2}} \leq  \varlimsup_{m\to \infty} \left(\frac{1}{m+1}  \sum_{k=0}^m \hat{A}_k^{4p}\right)^{\frac{1}{4}} \varlimsup_{m\to \infty} \left(\frac{1}{m+1}  \sum_{k=0}^m \hat{C}_k^4\right)^{\frac{1}{4}}\leq \hat{A}^p\hat{C};\\
		&&\varlimsup_{m\to \infty} \left(\frac{1}{m+1}  \sum_{k=0}^m \hat{C}_k^{2(p+1)}\right)^{\frac{1}{2}}= \hat{C}^{p+1}.
	\end{eqnarray*}
	Hence
	\begin{eqnarray*}
		&&\varlimsup_{m\to \infty} \Big(\frac{1}{m+1}  \sum_{k=0}^m \hat{C}^2_k G_k^2\Big)^{\frac{1}{2}} \leq
		\max \{8\hat{A}\hat{C}, 16K\hat{C}^2, 8^p\hat{A}^p\hat{C}, 16^pK^p \hat{C}^{1+p} \} = \hat{C}\hat{G}.
	\end{eqnarray*}
As a result
	\begin{eqnarray*}
		\varlimsup_{t\to \infty}\frac{1}{t}\log r(t) &\leq& -h_0 +  K(1+ 4\hat{G}) \hat{C} \Big( \Gamma(\omega,2) +  \Gamma(\omega,4)^{2}+\Gamma(\omega,2p+2)^{p+1}\Big)<0,
	\end{eqnarray*}
	due to \eqref{criterion} which proves the exponentially asymptotical stability of the zero solution of system \eqref{lin1}.
\end{proof}

\begin{corollary}\label{Phi}
	Consider the equation
	\begin{equation}\label{AutoLinear}
	dz(t) = Az(t) dt + Cz(t) d\omega(t)
	\end{equation}
	in which $A,C\in \R^{d\times d}$, $A$ is negative definite, i.e. there exists constant $ h_A>0$ such that 
	\begin{equation}\label{conditions1}
	\langle x, Ax \rangle \leq - h_A \|x\|^2,\quad \forall x\in \R^d. 
	\end{equation}
	Denote by $\Phi(t,\omega)$ the matrix solution of \eqref{AutoLinear}, $\Phi(0,\omega)=Id$. Then for any given $\delta>0$
	\begin{equation}\label{phiest}
	\|\Phi(t,\omega)\| \leq \exp \Big\{ -h_A t + \delta + \max\{\|C\|, \|C\|^{p}\} \kappa (t,\omega) \Big\},\quad  \forall t \in [0,1]
	\end{equation}
	where
	\begin{eqnarray}
	G &:=& \max\Big\{8 \|A\|,16K \|C\|, 8^p\|A\|^p,16^pK^p \|C\|^p\Big\},\label{GA} 
	\end{eqnarray}
	and 
	\begin{eqnarray}\label{kappa}
	\kappa (t,\omega) := \frac{1}{\delta^{p-1}} \ltn \omega \rtn_{p{\rm -var},[0,t]}^{p}+ 4K G \ltn \omega \rtn_{p {\rm - var},[0,t]} \Big(t + \ltn \omega \rtn_{p {\rm - var},[0,t]} + \ltn \omega \rtn_{p {\rm - var},[0,t]}^p\Big).
	\end{eqnarray}
\end{corollary}

\begin{proof}
	First, it can be seen that
	\[
	 \|C\| \ltn \omega \rtn_{p{\rm -var},[0,t]} \leq \delta + \frac{1}{\delta^{p-1}} \Big(\|C\| \ltn \omega \rtn_{p{\rm -var},[0,t]} \Big)^p
	\]
	For any $x_0\in \R^d$, it follows from \eqref{yest} that for any $t\in [0,1]$ and $y(t) = \frac{\Phi(t,\omega)x_0}{\|\Phi(t,\omega)x_0\|}$
	\begin{eqnarray*}
		\log \|\Phi(t,\omega)x_0\| &=& \int_0^t \langle y(s),Ay(s) \rangle ds + \int_0^t \langle y(s), C y(s) \rangle d\omega(s) \notag \\
		&\leq& -h_A t + \|C\| \ltn \omega \rtn_{p {\rm - var},[0,t]} + 2 K \|C\|  \ltn \omega \rtn_{p {\rm - var},[0,t]}  \ltn y \rtn_{q {\rm - var},[0,t]} \notag \\
		&\leq& - h_A t +\delta + \frac{1}{\delta^{p-1}} \Big(\|C\| \ltn \omega \rtn_{p{\rm -var},[0,t]} \Big)^{p} \\
		&& + 4K \|C\|  \ltn \omega \rtn_{p {\rm - var},[0,t]} G \Big(\max \{t,t^p\} + \ltn \omega \rtn_{p {\rm - var},[0,t]} + \ltn \omega \rtn_{p {\rm - var},[0,t]}^p\Big) \notag\\
		&\leq& - h_A t +  \delta + \frac{1}{\delta^{p-1}} \Big(\|C\| \ltn \omega \rtn_{p{\rm -var},[0,t]} \Big)^{p} \\
		&& + 4K G\|C\|  \ltn \omega \rtn_{p {\rm - var},[0,t]} \Big(t + \ltn \omega \rtn_{p {\rm - var},[0,t]} + \ltn \omega \rtn_{p {\rm - var},[0,t]}^p\Big)\\
		&\leq& - h_A t + \delta + \max\{\|C\|^{p}, \|C\|\} \kappa(t,\omega),
	\end{eqnarray*}
	which proves \eqref{phiest}. 
\end{proof}

\begin{remark}\label{comparison}
	(i), In \cite{GAKLBSch2010} and \cite{ducGANSch18} the authors develop the semigroup method to estimate the H\"older norm of $y$ on intervals $\tau_k, \tau_{k+1}$ where $\tau_k$ is a sequence of stopping times
		\[
		\tau_0 = 0, \tau_{k+1} - \tau_k + \ltn x \rtn_{\beta,[\tau_k,\tau_{k+1}]} = \mu
		\]
		for some $\mu \in (0,1)$ and $\beta > \frac{1}{p}$, which leads to the estimate of the exponent 
		\[
		- \Big(h_A - Q e^{h_A}\max\{C_f,\|C\|\} \frac{n}{\tau_n} \Big), 
		\]
		 where $h_A$ is given in \eqref{conditions1}, $C_f$ is the Lipchitz constant of $f$ and $Q$ is  generic constant independent of $A, f, C, \omega$. It is then proved that there exists $\liminf \limits_{n \to \infty} \frac{\tau_n}{n} = \frac{1}{d}$, where $d = d(\mu)$ depends on the moment of the stochastic noise. As such the rate of exponential convergence of the solution to zero can be estimated as 
		\begin{equation}\label{criterion2}
		-\Big(h_A - Q e^{h_A}\max\{C_f,\|C\|\} d\Big). 
		\end{equation}
		However, it is required from the stopping time analysis (see \cite[Section 4]{ducGANSch18}) that the stochastic noise has to be small in the sense that the moment of H\"older semi-norm $\ltn \omega \rtn_{\beta,[-1,1]}$ must be controlled as small as possible. On the other hand, when reduced to the case without noise, i.e. $C \equiv 0$, \eqref{criterion2} implies a very rough criterion for exponential stability of the ordinary differential equation
		\begin{equation}\label{oldcriterion}
		C_f \leq \frac{1}{Q} h_A e^{-h_A}.
		\end{equation}
By contrast, if $A,C$ are constant matrices and $f(t) \equiv C_f$, condition \eqref{criterion} is satisfied if
	\begin{eqnarray}\label{criterion1}
	h_A - C_f &>&K \|C\| (1 +4G)\Big\{\Gamma(\omega,2) +  \Gamma(\omega,4)^{2}+\Gamma(\omega,2+2p)^{1+p} \Big\},
	\end{eqnarray}
	where $G$ is given by \eqref{GA}. The left and the right hand sides of criteria \eqref{criterion} and \eqref{criterion1} therefore can be interpreted as, respectively, the decay rate of the drift term and the intensity of the volatility term. In this sense, criteria \eqref{criterion} and \eqref{criterion1} have the same form as the one below
	\begin{equation}\label{itocriterion}
	\liminf \limits_{t \to \infty} \frac{1}{t} \int_0^t [h(s)-f(s)]ds > \limsup \limits_{t \to \infty} \frac{1}{t} \int_0^t \|C(s)\|^2 ds
	\end{equation}
	for stochastic system driven by a standard Brownian motion (see e.g. \cite{mao}). Indeed, using Hypotheses (${\textbf H}_1$) -- (${\textbf H}_2$) and estimate \eqref{Itoest}, it follows that
	\begin{eqnarray*}
	d E \|x(t)\|^2 &=& E \Big(2 \langle x(t), A(t) x(t) \rangle +  2 \langle x(t), F(t,x(t)) \rangle + \|C(t) x(t)\|^2 \Big) dt \\
	&\leq& \Big(- h(t) + f(t) + \|C(t)\|^2 \Big) E  \|x(t)\|^2 
	\end{eqnarray*}
which then derives the exponential stability given \eqref{itocriterion}.\\ 
In addition, since $\|C\| (1 +4G)$ is an increasing function of $\|C\|$, criterion \eqref{criterion1} is satisfied in case the driving noise $\omega$ is small in the sense that the quantity in the brackets $\{\dots\}$ is small enough, or in case $\|C\|$ is small. 	
 Moreover, for ordinary differential equations, criteria \eqref{criterion} and \eqref{criterion1} reduce to $h_A>C_f$, which is the classical criterion and is much better than \eqref{oldcriterion} for dissipative systems. Therefore criteria \eqref{criterion} and \eqref{criterion1} can be viewed as a better generalization of the classical results on exponential stability for dissipative systems.

(ii), Regarding to system \eqref{AutoLinear}, we could have, in some special cases, better estimates than \eqref{phiest}. In particular, if $A$ and $C$ are commute, then a direct computation shows that
\begin{equation}\label{phiformula}
\Phi(t,\omega) = \exp \{ At + C \omega(t)\},\quad \forall t \geq 0.
\end{equation}
As a result,
\begin{equation}\label{phinormest}
\|\Phi(t,\omega)\| \leq \|e^{At}\| \|e^{C\omega(t)}\| \leq \|e^{At}\| e^{\|C\||\omega(t)|},\quad \forall t \geq 0.
\end{equation}
Therefore, under the assumption that
\[
\limsup \limits_{t \to \infty} \frac{\omega(t)}{t} = 0,
\]
(which is often satisfied for almost alls realization $\omega$ of a fractional Brownian motion), it follows that
\[
\limsup \limits_{t \to \infty} \frac{1}{t} \log \|\Phi(t,\omega)\| \leq \limsup \limits_{t \to \infty} \frac{1}{t} \log \|e^{At}\| + \limsup \limits_{t \to \infty} \frac{1}{t} \|C\| |\omega(t)| =  \limsup \limits_{t \to \infty} \frac{1}{t} \log \|e^{At}\|.   
\]
In this situation, the exponential stability criterion of system \eqref{AutoLinear} is then equivalent to the one of the autonomous ordinary differential equation $\dot{z} = A z$, which is equivalent to that $A$ has all eigenvalues with negative real parts. However, since \eqref{phiformula} does not hold in general, we could not obtain \eqref{phinormest} but only the discrete version \eqref{phiest}.

(iii), The strong condition \eqref{negdefh} is still able to cover several interesting cases, for instance if $A(t) \equiv A$ with negative real part eigenvalues. Then there exists a positive definite matrix $Q$, which is the solution of the matrix equation
\[
A^{\rm T} Q^2 + Q^2 A = 2D
\]
where $D$ is a symmetric negative definite matrix \cite[Chapter 2 \& Chapter 5]{burton} such that $\langle Dx,x \rangle \leq - \lambda_D \|x\|^2$. Under the transformation $\tilde{x} = Qx$ the system
\[
dx(t) = [A x(t) + F(t,x(t))]dt  + C(t)x(t) d\omega(t)
\]
will be tranformed to
\begin{eqnarray}
d \tilde{x}(t) &=& \Big[QAQ^{-1} \tilde{x}(t) + QF(t,Q^{-1}\tilde{x}(t)) \Big]dt + QC(t)Q^{-1}\tilde{x}(t)d\omega(t) \notag\\
&=& \Big[ \tilde{A} \tilde{x}(t) + \tilde{F}(t,\tilde{x}(t))\Big] dt + \tilde{C}(t)\tilde{x}(t) d\omega(t),
\end{eqnarray}
where $\tilde{F}$ is globally Lipschitz continuous with $f(t)$ in \eqref{lipF} is replaced by  $\tilde{f}(t) = \|Q\| \|Q^{-1}\| f(t)$; $\hat{A}, \hat{C}$ in \eqref{A-hat} and \eqref{C-hat} are replaced by $\|Q\| \|Q^{-1}\| \hat{A}, \|Q\| \|Q^{-1}\| \hat{C}$; and \eqref{negdefh} is of the form
\begin{eqnarray*}
\langle \tilde{x}, \tilde{A} \tilde{x} \rangle &=& \langle \tilde{x},  \frac{1}{2}[\tilde{A} + \tilde{A}^{\rm T}] \tilde{x}\rangle = \langle\tilde{x}, \frac{1}{2} [QAQ^{-1} + Q^{-1}A^{\rm T} Q] \tilde{x} \rangle \\
&=& \langle Q x,\frac{1}{2} QAQ^{-1}Qx\rangle + \langle Q x, \frac{1}{2} Q^{-1}A^{\rm T} Q^2 x\rangle\\
&=& \langle x,\frac{1}{2} [Q^2 A + A^{\rm T}Q^2]x \rangle \\
&=& \langle x, Dx \rangle \leq - \lambda_D \|x\|^2 \leq - \frac{\lambda_D}{\|Q\|^2} \|\tilde{x}\|^2.
\end{eqnarray*}
Therefore we are still able to apply Theorem \ref{stabfSDElin} with a small modification of conditions \eqref{negdef} and \eqref{criterion}.

(iv), It is important to note that for the nonautonomous situation, the semigroup generated from the method in \cite{ducGANSch18} or \cite{GAKLBSch2010} should be replaced by the two parameter flow $\Psi(t,s)$ generated from the nonautonomous differential equation $\dot{z} = A(t)z$. As a result, all $p-$variation norm estimates for such $\Psi$ would be quite complex to present. Our method however helps overcome this drawback by using Lyapunov type functions, as seen in the proof of Theorem \ref{stabfSDElin}.

\end{remark}

\section{Applications: Existence of random attractors}
In this section we would like to apply the main result to study the following system
\begin{equation}\label{fSDE}
dx(t) = [Ax(t) + f(x(t))]dt + Cx(t)dB^H(t), x(0)=x_0 \in \R^d, 
\end{equation}
where $B^H$ is an one dimensional fractional Brownian motion with Hurst index $H>\frac{1}{2}$; $A$ is negative definite and $f: \R^d \to \R^d$ is globally Lipschitz continuous, i.e. there exist contants $h_A, c_f >0$ such that 
\begin{equation}\label{conditions}
\langle x, Ax \rangle \leq - h_A \|x\|^2,\quad \|f(x) - f(y)\| \leq c_f \|x-y\|, \quad \forall x,y\in \R^d. 
\end{equation}
Given $\frac{1}{2}< \nu < H$ and any time interval $[0,T]$, almost sure all realizations $\omega(\cdot) = B^H(\cdot,\omega)$ belong to the H\"older space $C^{\nu\rm{-Hol}}([0,T],\R)$ (see e.g. \cite[Proposition 1.6]{nourdin}), thus system \eqref{fSDE} can be solved in the pathwise sense and admits a unique solution $x(t,\omega,x_0)$, according to Theorem \ref{existence}. Moreover, it is proved, e.g. in \cite{GAKLBSch2010} that, the solution generates a so-called {\it random dynamical system} defined by $\varphi(t,\omega)x_0 := x(t,\omega,x_0)$ on the probability space $(\Omega,\mathcal{F},\mP)$ equipped with a metric dynamical system $\theta$, i.e. $\theta_{t+s} = \theta_t \circ \theta_s$ for all $t,s \in \R$. Namely, $\varphi: \R \times \Omega \times \R^d \to \R^d$ is a measurable mapping which is also continuous in $t$ and $x_0$ such that the cocycle property
\[
\varphi(t+s,\omega) =\varphi(t,\theta_s \omega) \circ \varphi(s,\omega),\quad \forall t,s \in \R, 
\] 
is satisfied \cite{arnold}. It is important to note that, given the probability space as $\Omega = \cC_0(\R,\R)$ of continuous functions on $\R$ vanishing at zero, with the Borel sigma-algebra $\cF$, the Wiener shift $\theta_t \omega (\cdot) = \omega(t+ \cdot)-\omega(t)$ and the Wiener probability $\mP$, it follows from \cite[Theorem 1]{GASch} that one can construct an invariant probability measure $\mP^H = B^H \mP$ on the subspace $\cC^\nu$ such that $B^H \circ \theta = \theta \circ B^H$, and $\theta$ is ergodic. \\
Following \cite{arnold-schmalfuss},\cite{crauel-flandoli}, we call a set $\hat{M} =
\{M(\omega)\}_{\omega \in \Omega}$ a {\it random set}, if $\omega
\mapsto d(x|M(\omega))$ is $\cF$-measurable for each $x \in \R^d$, where $d(E|F) = \sup\{\inf\{d(x, y)|y \in F\} | x \in E\}$  for $E,F$ are nonempty subset of $\R^d$ and $d(x|E) = d(\{x\}|E)$.  
Given a continuous random dynamical system $\varphi$ on $\R^d$.  An {\it universe} $\cD$ is a family of random sets
which is closed w.r.t. inclusions (i.e. if $\hat{D}_1 \in \cD$ and
$\hat{D}_2 \subset \hat{D}_1$ then $\hat{D}_2 \in \cD$). In our setting, we define the universe $\cD$ to be a family of random sets $D(\omega)$
which is {\it tempered} (see e.g. \cite[pp. 164, 386]{arnold}), namely$D(\omega)$ belongs to the ball $B(0,\rho(\omega))$ for all $\omega\in\Omega$  where the radius $\rho(\omega) >0$ is a {\it tempered random varible}, i.e.
\begin{equation}\label{tempered}
\lim \limits_{t \to \pm \infty} \frac{1}{t} \log \rho(\theta_{t}\omega) =0.
\end{equation}
An
invariant random compact set $\mathcal{A}  \in \cD$ is called a {\it pullback random attractor} in $\cD$, if $\mathcal{A} $ attracts
any closed random set $\hat{D} \in \cD$ in the pullback sense,
i.e.
\begin{equation}\label{pullback}
\lim \limits_{t \to \infty} d(\varphi(t,\theta_{-t}\omega)
\hat{D}(\theta_{-t}\omega)| \mathcal{A} (\omega)) = 0.
\end{equation}
Similarly, $\mathcal{A} $ is called a {\it forward random attractor} in $\cD$, if $\mathcal{A}$ attracts
any closed random set $\hat{D} \in \cD$ in the forward sense, i.e.
\begin{equation*}\label{eq3.6}
\lim \limits_{t \to \infty} d(\varphi(t,\omega) \hat{D}(\omega)|
\mathcal{A}(\theta_t \omega)) = 0.
\end{equation*}
The existence of a random pullback attractor follows from the existence of a random pullback absorbing set
(see \cite{crauel-flandoli},\cite{schenk-hoppe01}). A random set $\mathcal{B}  \in \cD$ is called {\it pullback
	absorbing} in a universe $\cD$ if $\mathcal{B} $ absorbs all sets in
$\cD$, i.e. for any $\hat{D} \in \cD$, there exists a time $t_0 =
t_0(\omega,\hat{D})$ such that
\begin{equation}\label{absorb}
\varphi(t,\theta_{-t}\omega)\hat{D}(\theta_{-t}\omega) \subset
\mathcal{B} (\omega), \ \textup{for all}\  t\geq t_0.
\end{equation}
Given a universe $\cD$ and a random compact
pullback absorbing set $\mathcal{B} \in \cD$, there exists a unique random pullback attractor
(which is then a weak attractor) in $\cD$, given by
\begin{equation}\label{at}
\mathcal{A}(\omega) = \cap_{s \geq 0} \overline{\cup_{t\geq s} \varphi(t,\theta_{-t}\omega)\mathcal{B}(\theta_{-t}\omega)}. 
\end{equation}
The reader is referred to a survey on random attractors in \cite{crauelkloeden15}.

\begin{lemma}
	For $\delta> 0$, the function $\kappa$ defined in \eqref{kappa} satisfies \\
	(i) For all $0<s<t<1$
	\begin{equation}\label{kappa2}
	\kappa(t,\omega) \geq \kappa(s,\omega) + \kappa(t-s,\theta_s \omega),
	\end{equation}
	(ii) 
	For all $0\leq t'\leq 1$
	\begin{equation}\label{kappa3}
	\kappa(1,\theta_{t'}\omega)\leq 2^{p} [\kappa(1,\omega)+\kappa(1,\theta_{1}\omega)].
	\end{equation}
	(iii) $E\ \kappa(1,\omega)<\infty$.
\end{lemma}
\begin{proof}
	(i) The inequalitiy holds since $\ltn \omega \rtn_{p {\rm - var},[s,t]}^p, \ltn \omega \rtn_{p {\rm - var},[s,t]}^2$ and $ \ltn \omega \rtn_{p {\rm - var},[s,t]}^{p+1}$ are control functions (see \cite{friz} for details on control functions), meanwhile
	\[
	t\ltn \omega \rtn_{p {\rm - var},[0,t]} + s\ltn \omega \rtn_{p {\rm - var},[t,s+t]} \leq (t+s) \ltn \omega \rtn_{p {\rm - var},[0,t+s]}.
	\]	
	
	(ii) Due to  \cite[Lemma\ 2.1]{congduchong17} if $z$ is an arbitrary function of bounded $p-$variation on $[0,2]$ then
	\[
	\ltn z\rtn^p_{p\rm{-var},[0,2]}\leq 2^{p-1}(\ltn z\rtn^p_{p\rm{-var},[0,1]} +\ltn z\rtn^p_{p\rm{-var},[1,2]}),
	\]
	which implies that for all $n\geq 0$
	\begin{eqnarray*}\label{pvar}
		\ltn z\rtn^n_{p\rm{-var},[0,2]} &\leq& 2^{\frac{(p-1)n}{p}}(\ltn z\rtn^p_{p\rm{-var},[0,1]} +\ltn z\rtn^p_{p\rm{-var},[1,2]} )^{\frac{n}{p}} \notag\\
		&\leq& 2^{\max\{p, n\}-1}\left(\ltn z\rtn^n_{p\rm{-var},[0,1]} + \ltn z\rtn^n_{p\rm{-var},[1,2]}\right).
	\end{eqnarray*}
	Therefore, taking into account the formula \eqref{kappa} defining $\kappa$ we can easily derive \eqref{kappa3}.
	
	(iii) Recall that in this section we consider equation \eqref{fSDE}, hence  $\omega$ is a realization of a fractional Brownian motion $B^H(t,\omega)$. Observe that for $\nu=1/p<H$ and $t>0$ be arbirary, $ \ltn \omega \rtn_{p {\rm -var}, [0,t]}\leq t^{\nu} \ltn \omega \rtn_{\nu {\rm -Hol}, [0,t]}$.
	
	Fix $q_0\geq \max\{\frac{2}{H-\nu}, 2p+2\}, q_0\in \N$. Apply \cite[Corollary A2]{friz} for $\alpha=\nu+\frac{1}{q_0}$ and  \cite[Remark\ 1.2.2, p\ 7]{mishura} we get
	\begin{eqnarray*}
	E\ltn B^H(\cdot,\omega)\rtn^{q_0}_{\nu{\rm -Hol},[0,1]} 
	&\leq &\left( \frac{32(\nu+\frac{2}{q_0})}{\nu}\right)^{q_0} \int_0^{1}\int_0^1\frac{E\|B^H(u,\omega)-B^H(v,\omega)\|^{q_0}}{|u-v|^{\nu q_0+2}} dudv\\
	&\leq & \left( \frac{32(\nu+\frac{2}{q_0})}{\nu}\right)^{q_0} \int_0^{1}\int_0^1 \frac{2^{q_0/2}\Gamma(\frac{q_0+1}{2})}{\sqrt{\pi}}|u-v|^{(H-\nu-2/q_0)q_0} dudv\\
	&\leq & \left(32\sqrt{2(q_0+1)}\right)^{q_0} \frac{2}{[(H-\nu)q_0-1](H-\nu)q_0}\\
	&\leq & \left(32\sqrt{2(q_0+1)}\right)^{q_0},
	\end{eqnarray*}
	in which $\Gamma(n)$ is the Gamma function.	This implies
	\begin{equation}\label{beta1}
	\left(E\ltn B^H(\cdot,\omega)\rtn^{q_0}_{p{\rm -var},[0,1]}\right)^{\frac{1}{q_0}}\leq 32\sqrt{2(q_0+1)}=: \beta,
	\end{equation}
	and since $q_0 \geq 2p+2$ we conclude that
	\begin{equation}\label{Ekappa}
	E\  \kappa(1,\omega)   \leq  \max\{\frac{1}{\delta^{p-1}} ,4K G \} (\beta+\beta^p+\beta^2+ \beta^{p+1})<\infty
	\end{equation}
\end{proof}

Before stating the main result, we need the following results (the technical proofs are provided in the Appendix).

\begin{lemma}[Gronwall-type lemma] \label{lem2}
Assume that $z(\cdot), \alpha(\cdot): [a,b] \to \R^+$ satisfy
\begin{equation}\label{gronwall1}
z(t) \leq z_0 + \int_a^t \alpha(s)ds + \int_a^t \eta z(s) ds,\quad \forall t\in [a,b]
\end{equation}
for some $z_0, \eta>0$.
Then
\begin{equation}\label{gronwall2}
z(t) \leq z_0 e^{\eta (t-a)} + \int_a^t \alpha(s) e^{\eta(t-s)}ds,\quad \forall t\in [a,b].
\end{equation}	
\end{lemma}


\begin{lemma}\label{tempered}
Consider the random variable 
\begin{equation}\label{xi}
\xi(\omega):= 1+ \sum_{k=1}^{\infty} \exp \Big\{\Big(-h+c \frac{1}{k}\sum_{i=0}^{k-1}\kappa(1,\theta_{-i}\omega) \Big)k\Big\},
\end{equation}
where $h$, $c$ are given positive numbers and $\kappa$ is defined by \eqref{kappa}. Then there exists $\varepsilon>0$ such that if $c<\varepsilon$, $\xi(\omega)$ is tempered.
\end{lemma}
Given the universe $\cD$ of tempered random sets with property \eqref{tempered}, our second main result is then formulated as follows.

\begin{theorem}\label{attractor}
Assume that $h_A > c_f$. There exists an $\epsilon >0$ such that under condition $\|C\| < \epsilon$, $\varphi$ possesses a random pullback attractor consisting only of one random point $a(\omega)$ in the universe $\cD$ of tempered random sets. Moreover, every tempered random set converges to the random attractor in the pullback sense with exponential rate.
\end{theorem} 
\begin{proof}
We summarize the steps of the proof here. In {\bf Step 1} we prove \eqref{xest}, which helps to prove \eqref{xest1} in the forward direction and \eqref{xdiscrete} in the pullback direction, by choosing $\|C\|<\epsilon$ such that \eqref{attractorcriterion} is statisfied. As a result, there exists an absorbing set of the system which is a random ball with its radius described in \eqref{b}. The existence of the random attractor $\mathcal{A}$ is then followed. In {\bf Step 2}, we prove that any two different points $a_1,a_2$ in attractor $\mathcal{A}(\omega)$ can be pulled from fiber $\omega$ backward to fiber $\theta_{-t^*} \omega$, such that the difference of two solutions starting from fiber $\theta_{-t^*} \omega$ in fiber $\omega$ can be estimated by \eqref{y*}. Finally, using \eqref{btemper}, 
we conclude that $a_1(\omega) = a_2(\omega)$ almost surely, which proves that $\mathcal{A}$ is a single random point.  \\  

{\bf Step 1.} Fix a $\delta>0$ which will be specified later. We first show that there exists an absorbing set for system \eqref{fSDE}. Using \eqref{AutoLinear} and the method of variation of parameter as in \eqref{u}, one derives from \eqref{fSDE} the integral equation
\begin{equation*}
x(t,\omega,x_0) = \Phi(t,\omega)x_0 + \int_0^t \Phi(t-s,\theta_s \omega) f(x(s,\omega,x_0))ds,
\end{equation*}
where $\Phi$ defined in Corollary \ref{Phi}. 
Hence it follows from \eqref{phiest} and \eqref{kappa2} that for any $t \in [0,1]$
\allowdisplaybreaks
\begin{eqnarray*}
&&\|x(t,\omega,x_0)\| \\
&\leq& \| \Phi(t,\omega)x_0 \|+ \int_0^t \|\Phi(t-s,\theta_s \omega)\| \Big( c_f \|x(s,\omega,x_0)\| + \|f(0)\| \Big)ds \\
&\leq& \exp \Big\{ -h_A t + \delta + \max\{\|C\|, \|C\|^{p} \} \kappa (t,\omega) \Big\}\|x_0\| \\
&& + \int_0^t \exp \Big\{ -h_A (t-s) + \delta+\max\{\|C\|, \|C\|^{p} \}\kappa (t-s,\theta_s \omega) \Big\} 
\Big( c_f \|x(s,\omega,x_0)\| + \|f(0)\| \Big)ds \\
&\leq& \exp \Big\{ -h_A t +\delta + \max\{\|C\|, \|C\|^{p} \}\kappa (t,\omega) \Big\} \|x_0\| \\
&& + \int_0^t \exp \Big\{ -h_A (t-s) + \delta+ \max\{\|C\|, \|C\|^{p} \} \big[\kappa (t,\omega) -\kappa(s,\omega)\big] \Big\}\Big( c_f \|x(s,\omega,x_0)\| + \|f(0)\| \Big)ds. 
\end{eqnarray*}
Assign $z(t) :=  \|x(t,\omega,x_0)\| \exp \Big\{ h_A t - \max\{\|C\|, \|C\|^{p} \} \kappa(t,\omega)\Big \}$, then for any $t \in[0,1]$
\begin{eqnarray*}
z(t) &\leq& \|x_0\| e^{\delta} + \|f(0)\| e^{\delta} \int_0^t e^{h_A s - \max\{\|C\|, \|C\|^{p} \} \kappa(s,\omega)}ds  + \int_0^t c_f e^{\delta} z(s) ds,
\end{eqnarray*}
which has the form of \eqref{gronwall1}. By applying Gronwall lemma \ref{lem2}, we obtain
\begin{eqnarray*}
z(t) &\leq& \|x_0\| e^{\delta}  \exp \Big \{ c_fe^{\delta}t \Big\} + \|f(0)\|e^{\delta}   \int_0^t \exp \Big \{ c_f e^{\delta} (t-s) +h_A s - \max\{\|C\|, \|C\|^{p} \} \kappa(s,\omega)\Big\}ds, 
\end{eqnarray*}
for all $t\in [0,1]$. This follows that for any $t \in [0,1]$
\begin{eqnarray*}
&&\|x(t,\omega,x_0)\| \\
&\leq& \|x_0\|  \exp \Big \{ -h_A t+ \delta + \max\{\|C\|, \|C\|^{p} \} \kappa(t,\omega) +c_fe^{\delta}t \Big \} \notag \\
&&+ \|f(0)\|e^{\delta}   \int_0^t \exp \Big \{ c_f e^{\delta} (t-s) - h_A (t-s)+ \max\{\|C\|, \|C\|^{p} \} (\kappa(t,\omega) -\kappa(s,\omega))\Big\}ds \notag\\
&\leq& \|x_0\| \exp \Big \{ -\big(h_A - c_fe^{\delta} \big) t + \delta+ \max\{\|C\|, \|C\|^{p} \} \kappa(t,\omega)  \Big \}\notag \\
&&+ \|f(0)\|e^{\delta}  \int_0^t \exp \Big \{ -(h_A -c_f e^{\delta}) (t-s) + \max\{\|C\|, \|C\|^{p} \}(\kappa(t,\omega) -\kappa(s,\omega))\Big\}ds.
\end{eqnarray*}
Since $h_A > c_f$ there exists $\delta >0$ such that
\begin{eqnarray}\label{h3}
h := h_A - c_fe^{\delta} - \delta> 0.
\end{eqnarray}
Then for all $t\in [0,1]$
\allowdisplaybreaks
\begin{eqnarray}\label{xest}
&&\|x(t,\omega,x_0)\| \notag \\
&\leq&  \|x_0\| \exp \Big \{ -(h+\delta) t+ \delta + \max\{\|C\|, \|C\|^{p} \}\kappa(t,\omega)  \Big \}\notag \\
&&+ \|f(0)\|e^{\delta}   \int_0^t \exp \Big \{ -(h+\delta) (t-s)+ \max\{\|C\|, \|C\|^{p} \}(\kappa(t,\omega) -\kappa(s,\omega))\Big\}ds\notag\\
&\leq&  \|x_0\| \exp \Big \{ -(h+\delta)  t + \delta + \max\{\|C\|, \|C\|^{p} \} \kappa(1,\omega) \Big \}\notag\\
&&+ \|f(0)\|e^{\delta}  \int_0^t \exp \Big \{ -(h+\delta) (t-s)  + \max\{\|C\|, \|C\|^{p} \}\kappa(1,\omega)\Big\}ds\notag\\
&\leq& \|x_0\| \exp \Big \{ -(h+\delta)  t + \delta+ \max\{\|C\|, \|C\|^{p} \} \kappa(1,\omega) \Big \}+ \frac{\|f(0)\|}{h+\delta}\exp \Big\{\delta+\max\{\|C\|, \|C\|^{p} \}\kappa(1,\omega) \Big\},\notag \\  
\end{eqnarray}
as $\kappa$ is an increasing function of $t$.
In particular
\begin{eqnarray*}
\|x(1,\omega,x_0)\|
&\leq& \|x_0\| \exp \Big \{ -h + \max\{\|C\|, \|C\|^{p} \}\kappa(1,\omega) \Big \} \\
&&+ \frac{\|f(0)\|}{h+\delta}\exp \Big\{\delta + \max\{\|C\|, \|C\|^{p} \}\kappa(1,\omega)\Big\}  .
\end{eqnarray*}
Assign
\begin{eqnarray*}
\alpha(\omega) &:=& \exp \Big\{-h + \max\{\|C\|, \|C\|^{p} \} \kappa(1,\omega) \Big\},\\
\beta(\omega) &:=&\frac{\|f(0)\|}{h+\delta}\exp \Big\{\delta + \max\{\|C\|, \|C\|^{p} \} \kappa(1,\omega)\Big\}.
\end{eqnarray*}
By induction one can show that for any $n \geq 1$
\allowdisplaybreaks
\begin{eqnarray}\label{xest1}
&&\|x(n,\omega,x_0)\| \notag\\
&\leq&  \|x(n-1,\omega,x_0)\| \alpha(\theta_{n-1}\omega) + \beta(\theta_{n-1} \omega) \notag\\
&\leq& \ldots \notag\\
&\leq& \|x_0\| \prod_{k=0}^{n-1} \alpha(\theta_{k} \omega) + \sum_{k=0}^{n-1} \beta (\theta_{k} \omega)\prod_{i=k+1}^{n-1} \alpha(\theta_{i } \omega)\notag\\
&\leq& \|x_0\| \exp \Big\{\Big(-h+ \max\{\|C\|, \|C\|^{p} \} \frac{1}{n}\sum_{k=0}^{n-1}\kappa(1,\theta_{k}\omega) \Big)n \Big\}\notag\\
&&+ \sum_{k=0}^{n-1}\frac{\|f(0)\|}{h+\delta} e^{h+\delta} \exp \Big\{\Big(-h + \max\{\|C\|, \|C\|^{p} \} \frac{1}{n-k}\sum_{i=k}^{n-1}\kappa(1,\theta_{i }\omega)\Big)(n-k)\Big\}.
\end{eqnarray}
Using \eqref{xest} and \eqref{xest1}, we have for $t\in [(n, n+1]$ 
\begin{eqnarray}\label{xest2}
&&\|x(t,\omega,x_0)\|\notag\\ 
&\leq& \|x(n,\omega,x_0) \|\exp\Big\{-(h+\delta)(t-n) +\delta+\max\{\|C\|, \|C\|^{p} \}\kappa(1,\theta_{n}\omega) \Big\}\notag \\
&&+  \frac{\|f(0)\|}{h+\delta}\exp \big\{\delta +\max\{\|C\|, \|C\|^{p} \}\kappa(1,\theta_{n}\omega) \big\}\notag\\
&\leq & \|x_0\|e^{h+\delta}\exp \Big\{ \Big(-h +\max\{\|C\|, \|C\|^{p} \}\frac{1}{n+1}\sum_{k=0}^n\kappa(1,\theta_{k}\omega) \Big)(n+1)\Big\}\notag\\
&&+\sum_{k=0}^n\frac{\|f(0)\|e^{2(h+\delta)}}{h+\delta}\exp\Big\{ \Big(-h+\max\{\|C\|, \|C\|^{p} \}\frac{1}{n-k+1}\sum_{i=k}^n \kappa(1,\theta_{i}\omega)\Big)(n-k+1)\Big\}.\notag
\end{eqnarray}

By computation using  \eqref{kappa3} we obtain
\allowdisplaybreaks
\begin{eqnarray}\label{xcont1}
&&\|x(t,\theta_{-t}\omega,x_0)\| \notag\\
&\leq&  \|x_0\|e^{2h+\delta}\exp\Big\{ \Big(-h +2^{p+1}\max\{\|C\|, \|C\|^{p} \} \frac{1}{n+2}\sum_{k=0}^{n+1}\kappa(1,\theta_{-k}\omega) \Big)(n+2)\Big\}\notag\\
&&+\sum_{k=1}^{n+1}\frac{\|f(0)\|e^{3h+2\delta}}{h+\delta}\exp\Big\{ \Big(-h +2^{p+1}\max\{\|C\|, \|C\|^{p} \}\frac{1}{k+1}\sum_{i=0}^k \kappa(1,\theta_{-i}\omega)\Big)(k+1)\Big\}.\notag\\
\end{eqnarray}
Then for a fixed random set $\hat{D}(\omega) \in \cD$ with the corresponding ball $B(0,\rho(\omega))$ satisfying \eqref{tempered}, and for any random point $x_0(\theta_{-t}\omega) \in \hat{D}(\theta_{-t}\omega) $, we have
\allowdisplaybreaks
\begin{eqnarray}
&& \|x(t,\theta_{-t}\omega,x_0(\theta_{-t}\omega))\| \notag\\ 
&\leq& \|x_0(\theta_{-t}\omega)\| e^{2h+\delta}\exp\Big\{ \Big(-h+2^{p+1}\max\{\|C\|, \|C\|^{p} \}\frac{1}{n+2}\sum_{k=0}^{n+1}\kappa(1,\theta_{-k}\omega) \Big)(n+2)\Big\} \notag\\%
&&+  \frac{\|f(0)\|e^{3h+2\delta}}{h+\delta}  \sum_{k=1}^{n+1} \exp \Big\{\Big(-h +2^{p+1} \max\{\|C\|, \|C\|^{p} \} \frac{1}{k+1}\sum_{i=0}^{k}\kappa(1,\theta_{-i}\omega)\Big)(k+1)\Big\}\notag\\
&\leq& \rho(\theta_{-t}\omega) e^{2h+\delta} \exp\Big\{ \Big(-h+2^{p+1}\max\{\|C\|, \|C\|^{p} \}\frac{1}{n+2}\sum_{k=0}^{n+1}\kappa(1,\theta_{-k}\omega) \Big)(n+2)\Big\} + b(\omega),\notag
\end{eqnarray}
where 
\begin{eqnarray}\label{b}
b(\omega):= 1+ \frac{\|f(0)\|e^{3h+2\delta}}{h+\delta}  \sum_{k=1}^{\infty} \exp \Big\{\Big(-h+ 2^{p+1}\max\{\|C\|, \|C\|^{p} \} \frac{1}{k}\sum_{i=0}^{k-1}\kappa(1,\theta_{-i}\omega) \Big)k\Big\}.
\end{eqnarray}
Now we choose $\delta$ small enough such that \eqref{h3} holds and $C$ which satisfies 
\begin{equation}\label{attractorcriterion}
h = h_A - c_fe^{\delta} -\delta  > 2^{p+1}\max\{\|C\|, \|C\|^{p} \}   \max\{\frac{1}{\delta^{p-1}} ,4K G \} (\beta+\beta^p+\beta^2+ \beta^{p+1})
\end{equation}
and set $\lambda :=  h-2^{p+1}\max\{\|C\|, \|C\|^{p} \} E\  \kappa(1,\cdot)$.  There exists $n_0=n(\omega)$ such that  
$$\exp\Big\{ \Big(-h +2^{p+1}\max\{\|C\|, \|C\|^{p} \} \frac{1}{n+2}\sum_{k=0}^{n+1}\kappa(1,\theta_{-k}\omega) \Big)(n+2)\Big\}\leq e^{\frac{-\lambda (n+2)}{2}}$$
 for all $n\geq n_0$
 and $\rho(\theta_{-t}\omega) \leq e^{\frac{\lambda t}{4}}$ for all $t\geq n_0$ due to \eqref{tempered}. This follows that
\begin{eqnarray}\label{xdiscrete} \|x(t,\theta_{-t}\omega,x_0(\theta_{-t}\omega))\|&\leq& 2 b(\omega), \; 
\end{eqnarray}
for $n$ large enough and uniformly in random points $x_0(\omega) \in \hat{D}(\omega)$. 
This proves \eqref{absorb} and there exists a compact absorbing set $\mathcal{B}(\omega) = \overline{B}(0,2b(\omega))$ for system \eqref{fSDE}. Due to Lemma \ref{tempered} $b(\omega)$ is tempered when $\|C\|$ is small enough and thus $\mathcal{B} \in \cD$, this prove the existence of a random attractor $\mathcal{A}(\omega)$ of the form \eqref{at} for system \eqref{fSDE}. 

{\bf Step 2.} Assume that there exist two different points $a_1(\omega), a_2(\omega) \in \mathcal{A}(\omega)$. Fix $t^* \in [m, m+1]$ and put $\omega^*=\theta_{-t^*}\omega$ and consider the equation
\begin{equation}\label{equ.star}
dx(t) = [Ax(t) + f(x(t))]dt + Cx(t)d\omega^*(t).
\end{equation}
Note that \eqref{omegagrowth} holds for $\omega^*$.
By the invariance principle there exist two different points $b_1(\omega^*), b_2(\omega^*) \in \mathcal{A}(\omega^*)$ such that
\[
a_i(\omega) = x(t^*,\omega^*,b_i), \quad i = 1,2.
\]
Put $ y(t,\omega^*):= x(t,\omega^*,b_1)- x(t,\omega^*,b_2)$ then $y(t^*,\omega^*) =a_1(\omega)- a_2(\omega)$
and we have 
\begin{eqnarray*}
dy(t,\omega^*)&=& [Ay(t,\omega^*) + F(t,y(t,\omega^*))]dt + Cy(t,\omega^*)d\omega^*(t)
\end{eqnarray*}
where $F(t,y) = f(y + u(t)) -  f(u(t))$, where $u(t)=x(t,\omega^*,b_2)$ satisfies also globally linear growth \eqref{lipF} with coefficient $c_f$ and condition $F(t,0) \equiv 0$.\\
Now repeating the calculation in Theorem \ref{stabfSDElin} in which $\omega$ is replaced by $\omega^*$, we obtain 
\begin{eqnarray*}
\frac{1}{t^*}\log \| y(t^*,\omega^*)\|&\leq & \frac{1}{t^*}\log \| y(0,\omega^*)\| - (h_A-c_f) +\frac{K\|C\|}{m}\sum_{k=0}^m\ltn\omega^*\rtn_{p{\rm -var},\Delta_k}\\
&&+\frac{4K\|C\|G}{m} \Big(\sum_{k=0}^m\ltn\omega^*\rtn_{p{\rm -var},\Delta_k} + \sum_{k=0}^m\ltn\omega^*\rtn^{2}_{p{\rm -var},\Delta_k}+\sum_{k=0}^m\ltn\omega^*\rtn^{p+1}_{p{\rm -var},\Delta_k}\Big),
\end{eqnarray*}
in which $G$ given in \eqref{GA}. 
Using the fact that $\mathcal{A} \subset \mathcal{B}$, we have $\|y(0,\omega^*)\| \leq 4 b(\omega^*)$. Now letting $\N \ni t^* = m \to \infty$ and using \eqref{btemper}, we obtain
\begin{eqnarray}\label{y*}
&&\varlimsup_{t^*\to\infty}\frac{1}{t^*}\log \| y(t^*,\omega^*)\|\notag \\
&\leq & - (h_A-c_f) +K\|C\|(1+4G) \Big(E\ltn \omega\rtn_{p\rm{-var},[0,1]}
 +E\ltn \omega\rtn^{2}_{p\rm{-var},[0,1]}+E\ltn \omega\rtn^{p+1}_{p\rm{-var},[0,1]}\Big)\notag\\
&\leq&- (h_A-c_f)+K(1+4G)\|C\| \left( \beta+\beta^{2}+\beta^{p+1}\right),
\end{eqnarray}
in which $\beta$ is given by \eqref{beta1}. 
Hence, there exists $\eps>0$ such that if we choose $\|C\|<\eps$ then $y(t^*,\theta_{-t^*}\omega)$ converges to zero exponentially. Hence $a_1(\omega) - a_2(\omega) \to 0$ which is a contradiction. This proves that $\mathcal{A}(\omega) \equiv \{a(\omega)\}$ is a single random point. Finally similar arguments then prove that
$\|x(t,\theta_{-t}\omega,x_0(\theta_{-t}\omega))-a(\omega)\|$ converges to $0$ as $t\to \infty$ in an exponential rate and uniformly in random points $x_0(\omega)$ in a tempered random set $\hat{D}(\omega) \in \cD$, which proves the last conclusion of Theorem \ref{attractor}.
\end{proof}	

\begin{example}[Stochastic SIR model]
	Following \cite{caraballo2018}, consider a stochastic version of "susceptible-infected-recovered" epidemic model (SIR)
	\begin{eqnarray}\label{stochSIR}
	dS_t &=& \Big[q - a S_t + b I_t - \gamma \frac{S_tI_t}{S_t + I_t +R_t}\Big] dt + \sigma_1 S_t dB^H_t \notag\\
	dI_t &=& \Big[- (a+b+c) I_t +  \gamma \frac{S_tI_t}{S_t+ I_t +R_t} \Big] dt + \sigma_2 I_t dB^H_t \notag\\
	dR_t &=& \Big[c I_t - a R_t\Big] dt + \sigma_3 R_t dB^H_t,
	\end{eqnarray}
	where $q,a,b,c,\gamma, \sigma_1, \sigma_2, \sigma_3 \geq 0$. System \eqref{stochSIR} can be rewritten in the following form of variable $y = (S,I,R)^{\rm T} \in \R^3$
	\begin{eqnarray}\label{stochSIR2}
	d y_t &=& [A y_t + F(y_t)]dt + Cy_t dB^H_t \notag\\
	&=& \left[\left(\begin{matrix}
		-a & b & 0\\ 0& -a-b-c & 0 \\ 0 & c & -a
	\end{matrix} \right)y_t + \left(\begin{matrix}
	q-\gamma \frac{S_tI_t}{S_t+I_t+R_t}\\ \gamma \frac{S_tI_t}{S_t+I_t+R_t} \\ 0 
\end{matrix} \right)\right] dt + \left(\begin{matrix}
\sigma_1 & 0 & 0\\ 0& \sigma_2 & 0 \\ 0& 0 & \sigma_3 
\end{matrix} \right) y_t dB^H_t.
	\end{eqnarray}
It is easy to check that 
\[
\|F(y_1) - F(y_2) \| \leq \gamma \Big(|S_1-S_2|+|I_1 - I_2| + |R_1-R_2|\Big) \leq \gamma\sqrt{3} \|y_1 - y_2\|,\quad  \forall y_1, y_2 \in \R^3_+,
\]
hence $F$ is globally Lipschitz continuous. The existence and uniqueness, as well as the positiveness of the solution of \eqref{stochSIR} are investigated in \cite{caraballo2018} using fractional calculus for Young integral \cite{zahle, zahle2}.  \\
To study the asymptotic behavior of system \eqref{stochSIR}, observe from \cite{caraballo2018} that $A$ is diagonalizable, which can be written in the form
\[
A = P D P^{-1}, \quad D = \left(\begin{matrix}
-a & 0 & 0\\ 0& -a & 0 \\ 0 & 0 & -a-b-c
\end{matrix} \right), \quad P = \left(\begin{matrix}
1 & 0 & \frac{b}{b+c}\\ 0& 0 & -1 \\ 0 & 1 & \frac{c}{b+c}
\end{matrix} \right), \quad P^{-1} = \left( \begin{matrix} 1& \frac{b}{b+c}&0 \\ 0&  \frac{c}{b+c} & 1 \\ 0&-1&0
\end{matrix} \right).
\] 
Therefore, by assigning $x := P^{-1} y$ and applying the integration by parts for Young system, we obtain the equation for $x$ as follows 
\begin{eqnarray}\label{stochSIR3}
dx_t &=& \Big[P^{-1} A P x_t + P^{-1}F(Px_t)\Big] dt + P^{-1}CPx_t dB^H_t \notag\\
&=& [D x_t + F_1(x_t)]dt + P^{-1}CP x_t dB^H_t,
\end{eqnarray}
which has the form of \eqref{fSDE} with
\[
\langle x, Dx \rangle \leq - a \|x\|^2,\quad \|F_1(x_1) - F_1(x_2)\| \leq \gamma \sqrt{3} \|P\| \|P^{-1}\| \|x_1-x_2\| \leq 4\sqrt{3}\gamma \|x_1-x_2\|, \quad \forall x_1,x_2\in \R^3.
\]
We are now in the situation to apply Theorem \ref{attractor} provided that condition \eqref{attractorcriterion} is satisfied, i.e. $a -4\sqrt{3} \gamma$, $\delta>0$ such that $a -4\sqrt{3} \gamma e^\delta - \delta > 0$, and $\sigma_{\rm max}:=\max \{\sigma_1, \sigma_2,\sigma_3\}  \geq 0$ small enough such that 
\begin{eqnarray}
a -4\sqrt{3} \gamma e^\delta - \delta \geq 2^{p+1} \max\{4 \sigma_{\rm max}, (4 \sigma_{\rm max})^{p}\}  \max\{\frac{1}{\delta^{p-1}} ,4K G \} (\beta+\beta^p+\beta^2+ \beta^{p+1}).
\end{eqnarray}
Under this condition, there exists an one-point pullback attractor for the tranformed system \eqref{stochSIR3} and thus for the original system \eqref{stochSIR} after the transformation $y = Px$.
\end{example}

\section{Appendix}
\begin{proof}[Proof of Lemma \ref{lem2}]
From \eqref{gronwall1} it follows that
\begin{eqnarray*}
d \Big(e^{-\eta t} \int_a^t z(s) ds\Big)& = &e^{-\eta t} \Big(-\eta \int_a^t z(s) ds + z(t)\Big)\leq e^{-\eta t} \Big(z_0+ \int_a^t \alpha(s)ds\Big),\quad \forall t\in [a,b].
\end{eqnarray*}
As a result
\[
\int_a^t z(s) ds \leq 	\int_a^t e^{\eta (t-s)} \Big(z_0+ \int_a^s \alpha(u)du\Big)ds.
\]
Hence combining with \eqref{gronwall1} and using the integration by parts one gets
\begin{eqnarray*}
z(t) &\leq& z_0 + \int_a^t \alpha(s)ds + \eta 	\int_a^t e^{\eta (t-s)} \Big(z_0+ \int_a^s \alpha(u)du\Big)ds \\
&\leq & z_0 e^{\eta (t-a)} + \int_a^t \alpha(s)ds - e^{\eta t} \int_a^t \Big(\int_a^s \alpha(u)du\Big) d (e^{-\eta s}) \\
&\leq & z_0  e^{\eta (t-a)} + \int_a^t e^{\eta(t-s)} \alpha(s)ds,
\end{eqnarray*}
which proves \eqref{gronwall2}.\\
\end{proof}

\begin{proof}[Proof of Lemma \ref{tempered}]
Firstly, since the dynamical system $\theta$ is ergodic in $(\Omega,\mathcal{F},\mathbb{P})$, for almost all $\omega \in \Omega$
\begin{eqnarray}\label{lambda}
\lim \limits_{k \to \infty} \left(-h+c \frac{1}{k}\sum_{i=0}^{k-1}\kappa(1,\theta_{-k}\omega) \right)&=& -h+c E\  \kappa(1,\cdot) \notag\\
&\leq&-h+c \max\{\frac{1}{\delta^{p-1}} ,4K G \} (\beta+\beta^p+\beta^2+ \beta^{p+1})
\end{eqnarray}
due to \eqref{Ekappa}.
Set  $-\lambda :=  -h+c E\  \kappa(1,\cdot)$. Take and fix a small positive number $\varepsilon$ such that
\begin{equation}\label{criterion3}
h >\varepsilon  \max\{\frac{1}{\delta^{p-1}} ,4K G \} (\beta+\beta^p+\beta^2+ \beta^{p+1}). 
\end{equation}
Then for any $0<c<\varepsilon$ we have $\lim \limits_{k \to \infty} \left(-h+c \frac{1}{k}\sum_{i=0}^{k-1}\kappa(1,\theta_{-k}\omega) \right)= -\lambda < 0$.

Consequently, the series
\[
\sum_{k=1}^{\infty} \exp \Big\{\Big(-h+c \frac{1}{k}\sum_{i=0}^{k-1}\kappa(1,\theta_{-i}\omega) \Big)k\Big\}
\] 
converges or $\xi(\omega)$ is finite for almost all $\omega\in \Omega$.

Next we are going to prove that $\xi(\omega)$ is tempered if $c$ is small enough. Using \eqref{kappa3}, it suffices to prove that
\begin{equation}\label{btemper}
\lim_{t\to\pm\infty\atop t\in\Z}\frac{1}{t}\log \Big[\xi(\theta_{t}\omega)\Big] = 0
\end{equation}
whenever $c<\varepsilon$. Indeed, replacing $\omega$ by $\theta_{-m}\omega$ where $m\in\Z^+$ in \eqref{xi} we get 
\begin{eqnarray*}
&&\xi(\theta_{-m}\omega) \\
&=& 1+ \sum_{k=1}^{\infty} \exp \Big\{-h k+c \sum_{i=0}^{k-1}\kappa(1,\theta_{-(i+m)}\omega) \Big\}.\\
&=& 1+\sum_{k=1}^{\infty} \exp \Big\{-hk+c \sum_{i=m}^{k+m-1}\kappa(1,\theta_{-i}\omega) \Big\}.\\
&= & 1+ \exp\left\{\left(h- c \frac{1}{m}\sum_{i=0}^{m-1}\kappa(1,\theta_{-i}\omega)\right)m \right\}  \sum_{k=1}^{\infty} \exp \left\{\left(-h+c\frac{1}{k+m} \sum_{i=0}^{k+m-1}\kappa(1,\theta_{-i}\omega)\right) (k+m)\right\}.\\
\end{eqnarray*}
By \eqref{lambda}, for each $N\in \N^*$, $\frac{1}{N}<\lambda$ there exists $n(\omega,N)$ such that  for all $n>n(\omega,N)$
$$
-\lambda-\frac{1}{N}\leq -h +c \frac{1}{n}\sum_{k=0}^{n-1}\kappa(1,\theta_{-k}\omega) \leq -\lambda +\frac{1}{N},\,
$$
and 
$$
-\lambda-\frac{1}{N}\leq -h+c \frac{1}{n}\sum_{k=0}^{n-1}\kappa(1,\theta_{k}\omega) \leq -\lambda +\frac{1}{N}.
$$
Therefore, with  $N$, $\omega$ fixed, if $m>n(\omega,N)$ we have
\begin{eqnarray*}
1&\leq& \xi(\theta_{-m}\omega)\leq  1+  e^{(-\lambda+\frac{1}{N}) m} \sum_{k=1}^{\infty} \exp\{(-\lambda +1/N)(k+m)\}\leq (D+1) e^{2m/N},
\end{eqnarray*}
where $D=  \sum_{k=1}^{\infty} \exp\{(-\lambda +\frac{1}{N})k\} <\infty$ and $D$ is independent of $m$.  Hence, it follows that
\[
0\leq \varlimsup_{m\to+\infty\atop m\in \Z}\frac{1}{m}\log \Big[\xi(\theta_{-m}\omega)\Big] \leq \lim_{m\to\infty}\frac{2m}{m N} = \frac{2}{N},
\]
for any $N$ large enough, which proves \eqref{btemper} for the case $t\to -\infty$.  

Similarly,  replacing $\omega$ by $\theta_{m}\omega$ where $m\in\Z^+$ in \eqref{xi} we obtain
\allowdisplaybreaks
\begin{eqnarray*}
\xi(\theta_{m}\omega)
&=& 1+  \sum_{k=1}^{\infty} \exp \Big\{-hk+c \sum_{i=0}^{k-1}\kappa(1,\theta_{-i+m}\omega) \Big\}.\\
&=& 1+  \sum_{k=1}^{m} \exp \Big\{-hk+c \sum_{i=0}^{k-1}\kappa(1,\theta_{-i+m}\omega) \Big\}+ \sum_{k=m+1}^{\infty} \exp \Big\{-hk+\|C\| \sum_{i=0}^{k-1}\kappa(1,\theta_{-i+m}\omega) \Big\},\\
\end{eqnarray*}
in which the second term is
\allowdisplaybreaks
\begin{eqnarray*}
&&\sum_{k=1}^{m} \exp \Big\{-hk+c \sum_{i=0}^{k-1}\kappa(1,\theta_{-i+m}\omega) \Big\}\\
&&= \sum_{k=1}^{m} \exp \Big\{-hk+c \sum_{i=m-k+1}^{m}\kappa(1,\theta_{i}\omega) \Big\}\\
&=&\exp \Big\{-h(m+1)+c \sum_{i=0}^{m}\kappa(1,\theta_{i}\omega) \Big\}  \sum_{k=1}^{m}  \exp \Big\{hk-c \sum_{i=0}^{k-1}\kappa(1,\theta_{i}\omega) \Big\}\\
&=&\exp \Big\{-h(m+1)+c \sum_{i=0}^{m}\kappa(1,\theta_{i}\omega) \Big\}\\
&&\times \left( \sum_{k=1}^{n(\omega,N)}  \exp \Big\{hk -c \sum_{i=0}^{k-1}\kappa(1,\theta_{i}\omega) \Big\}+ \sum_{k=n(\omega,N)+1}^{m}  \exp \Big\{hk-c \sum_{i=0}^{k-1}\kappa(1,\theta_{i}\omega) \Big\}\right)\\
&\leq& \exp\{(-\lambda+\frac{1}{N})(m+1)\} \times\\
&&\times  \left(  \sum_{k=1}^{n(\omega,N)}  \exp \Big\{hk- c \sum_{i=0}^{k-1}\kappa(1,\theta_{i}\omega) \Big\} + \sum_{k=n(\omega,N)+1}^{m}   \exp\{(\lambda+\frac{1}{N})k\} \right) \\
&\leq & \exp\{(-\lambda+\frac{1}{N})(m+1)\} \left(  \sum_{k=1}^{n(\omega,N)}  \exp \Big\{hk- c \sum_{i=0}^{k-1}\kappa(1,\theta_{i}\omega) \Big\}+ \frac{  \exp\{(\lambda+\frac{1}{N})(m+1)\} }{e^{\lambda+\frac{1}{N}}-1}\right) \\
&\leq&  e^{\frac{2}{N}(m+1)}D(\omega),
\end{eqnarray*}
where 
\[
D(\omega) =  \sum_{k=1}^{n(\omega,N)}  \exp \Big\{hk-c \sum_{i=0}^{k-1}\kappa(1,\theta_{i}\omega) \Big\}\ + \frac{1}{e^{\lambda+\frac{1}{N}}-1} \] 
and $m>n(\omega,N)$.\\
On the other hand, the third term is
\allowdisplaybreaks
\begin{eqnarray*}
&&\sum_{k=m+1}^{\infty} \exp \Big\{-hk+c \sum_{i=0}^{k-1}\kappa(1,\theta_{-i+m}\omega) \Big\}\\
&&=\sum_{k=m+1}^{\infty} \exp \Big\{-hk+ c \sum_{i=0}^{m-1}\kappa(1,\theta_{-i+m}\omega)+ c \sum_{i=m}^{k-1}\kappa(1,\theta_{-i+m}\omega) \Big\}\\
&&= \exp \Big\{-hm+c \sum_{i=1}^{m}\kappa(1,\theta_{i}\omega)\Big\}  \sum_{k=1}^{\infty} \exp \Big\{-hk+c \sum_{i=0}^{k-1}\kappa(1,\theta_{-i}\omega) \Big\}\\
&&\leq  e^{(-\lambda+\frac{1}{N})m} \times \sum_{k=1}^{\infty} \exp \Big\{-hk+ c \sum_{i=0}^{k-1}\kappa(1,\theta_{-i}\omega) \Big\}
\end{eqnarray*}
when $m>n(\omega,N)$.
To sum up, for $m>n(\omega,N)$ we have
\begin{eqnarray*}
1&\leq\xi(\theta_{m}\omega)&\leq 1+ e^{\frac{2}{N}(m+1)}D(\omega) + e^{(-\lambda+\frac{1}{N})m} \xi(\omega)\leq e^{\frac{2}{N}(m+1)} \left( 1+ D(\omega) + \xi(\omega)\right).
\end{eqnarray*}
Since $D(\omega), \xi(\omega)$ are independent of $m$, $\varlimsup \limits_{m\to+\infty\atop m\in \Z} \frac{\xi(\theta_m\omega)}{m}\leq \frac{2}{N}$ for any $N$ large enough,  we conclude that $\xi(\omega)$ is tempered.
\end{proof}

\section*{Acknowledgment}
This work was partially sponsored by the Max Planck Institute for Mathematics in the Science (MIS-Leipzig) and  also by Vietnam National Foundation for Science and Technology Development (NAFOSTED) under grant number FWO.101.2017.01.


\begin{thebibliography}{1}
%
\bibitem{arnold}
L. Arnold.
\newblock{\em Random Dynamical Systems.}
\newblock{Springer, Berlin Heidelberg New York}, 1998.
%
\bibitem{arnold-schmalfuss}
L. Arnold, B. Schmalfuss.
\newblock{\em Lyapunov's second method for random dynamical systems.}
\newblock{\em J. Differential Equations}\ {\bf 177}, No. 1, (2001), 235--265.
%
\bibitem{burton}
T. A. Burton.
\newblock{\em Volterra integral and differential equations.}
\newblock{Mathematics in Science and Engineering, Edited by C.K. Chui, Stanford University}, Vol. {\bf 202}, 2005.
%
\bibitem{caraballo2018}
T. Caraballo, S. Keraani.
\newblock{\em Analysis of a stochastic SIR model with fractional Brownian motion.}
\newblock{\em Stochastic Analysis and Applications,}\ {\bf 36} (5), (2018), 895--908.
%
\bibitem{congduchong17}
N. D. Cong, L. H. Duc, P. T. Hong.
\newblock{\em Young differential equations revisited.}
\newblock{J. Dyn. Diff. Equat.}, Vol. {\bf 30}, Iss. 4, (2018), 1921--1943. 
%
\bibitem{congduchong18}
N. D. Cong, L. H. Duc, P. T. Hong.
\newblock{\em Lyapunov spectrum for nonautonomous linear Young differential equations.}
\newblock{preprint, arXiv:1807.02680}.
%
\bibitem{crauelkloeden15}
H. Crauel, P. Kloeden.
\newblock{\em Nonautonomous and random attractors.}
\newblock{Jahresbericht Dtsch. Math-Ver}, {\bf 117}, (2015), 173--206.
%
\bibitem{crauel-flandoli}
H.~Crauel, F.~Flandoli,
\newblock{\em Attractors for random dynamical systems.}
\newblock{Probab. Theory Related Fields} {\bf 100}, No. 3,  (1994), 365--393.
%
\bibitem{demidovich}
B. P. Demidovich.
\newblock {\em Lectures on Mathematical Theory of Stability }.
\newblock {\em Nauka} (1967). In Russian.
%
\bibitem{ducetal06}
L. H. Duc, A. Ilchmann, S. Siegmund, P. Taraba.
\newblock{\em On stability of linear time varying second-order differential equations}
\newblock{ Quarterly of Applied Mathematics}, {\bf 64} (1), (2006), 137--151.
%
\bibitem{ducGANSch18}
L. H. Duc, M. J. Garrido-Atienza, A. Neuenkirch, B. Schmalfu\ss.
\newblock{\em Exponential stability of stochastic evolution equations driven by small fractional Brownian motion with Hurst parameter in $(\frac{1}{2},1)$}.
\newblock{ J. Differential Equations}, {\bf 264}, (2018), 1119-1145.
%
\bibitem{friz}
P. Friz, N. Victoir.
\newblock {\em Multidimensional stochastic processes as rough paths: theory and applications.}
\newblock {Cambridge Studies in Advanced Mathematics, 120. Cambridge Unversity Press, Cambridge}, 2010.
%
\bibitem{GAKLBSch2010}
M. Garrido-Atienza, B. Maslowski, B. Schmalfu\ss.
\newblock{\em Random attractors for stochastic equations driven by a fractional Brownian motion.}
\newblock{International Journal of Bifurcation and Chaos}, Vol. {\bf 20}, No. 9, (2010), 2761--2782.
%
\bibitem{garrido-atienzaetal}
M. Garrido-Atienza, A. Neuenkirch, B. Schmalfu\ss.
\newblock{\em Asymptotic stability of differential equations driven by H\"older-continuous paths}
\newblock{J. Dyn. Diff. Equat.}, Vol. {\bf 30}, Iss. 1, (2018), 359--377.
%
\bibitem{GASch}
M. Garrido-Atienza, B. Schmalfuss.
\newblock{\em Ergodicity of the infinite dimensional fractional Brownian motion.}
\newblock{J. Dyn. Diff. Equat.}, {\bf 23}, (2011), 671--681. DOI 10.1007/s10884-011-9222-5.
%
\bibitem{GABSch18}
M. Garrido-Atienza, B. Schmalfuss.
\newblock{\em Local Stability of Differential Equations Driven by H\"older-Continuous Paths with H\"older Index in $(\frac{1}{3},\frac{1}{2})$.}
\newblock{SIAM J. Appl. Dyn. Syst.} Vol. {\bf 17}, No. 3, (2018), 2352--2380.	
%
\bibitem{hale}
J. K. Hale, H. Kocak.
\newblock{\em Dynamics and bifurcations.}
\newblock{Springer-Verlag New York}, (1991).
%
\bibitem{khasminskii}
R. Khasminskii.
\newblock{\em Stochastic stability of differential equations.}
\newblock{Springer, Vol. 66}, (2011).
%
\bibitem{mandelbrot}
B. Mandelbrot, J. van Ness.
\newblock{\em Fractional Brownian motion, fractional noises and applications.}
\newblock{SIAM Review}, {\bf 4}, No. 10, (1968), 422--437.
%
\bibitem{mao}
X. Mao,
\newblock{\em Stochastic differential equations and applications.}
\newblock{Elsevier}, (2007).
%
\bibitem{mishura}
Y. Mishura.
\newblock{\em Stochastic calculus for fractional Brownian motion and related processes.}
\newblock{Lecture notes in Mathematics, Springer}, (2008).
%
\bibitem{nemytskii}
V. V. Nemytskii, V. V Stepanov.
\newblock {\em Qualitative theory of differential equations.}
\newblock {Princeton University Press,} (1960).
%
\bibitem{nualart}
D. Nualart, A. R{\u{a}}{\c{s}}canu.
\newblock {\em Differential equations driven by fractional {B}rownian motion}.
\newblock {Collect. Math.} {\bf 53}, No. 1, (2002), 55--81.
%
\bibitem{nourdin}
I. Nourdin.
\newblock{\em Selected aspects of fractional Brownian motion.}
\newblock{Bocconi University Press, Springer}, (2012).
%
\bibitem{schenk-hoppe01}
K.~R.~Schenk-Hoppe,
\newblock{\em Random attractors-General properties, existence and applications to stochastic bifurcation theory.}
\newblock{Discrete Contin. Dyn. Syst.} {\bf 4}, No. 1, (1998), 99--130.
%
\bibitem{Shiryaev}
A.~N.~{Shiryaev},
\newblock{\em Probability}.
\newblock{Springer-Verlag, New York}, (1996).
%
\bibitem{wazewski}
T. Wazewski
\newblock{\em Sur la limitation des int\'egrales des syst\'emes d'\'equations diff\'erentielles lin\'eaires ordinaries}.
\newblock{ Studia Mathematica}. {\bf 10}, (1948), 48-59.
\bibitem{young}
L.C. Young.
\newblock{\em An integration of H{\"o}lder type, connected with Stieltjes integration.}
\newblock{ Acta Math.} {\bf 67}, (1936), 251--282.
%
\bibitem{zahle}
M. Z{\"a}hle.
\newblock {\em Integration with respect to fractal functions and stochastic calculus. {I}.}
\newblock { Probab. Theory Related Fields.} {\bf 111}, No. 3, (1998), 333--374.
%
\bibitem{zahle2}
M. Z{\"a}hle.
\newblock {\em Integration with respect to fractal functions and stochastic calculus. {II}.}
\newblock { Math. Nachr.} {\bf 225}, (2001), 145--183.
\end{thebibliography}
\end{document}